\documentclass[11pt,final]{nyjm}
\usepackage{amsmath}
\usepackage[notref,notcite,color]{showkeys}
\usepackage{rcs}
\usepackage{amssymb}

\usepackage{xfrac}
\usepackage[all,2cell]{xy}
\UseComputerModernTips
\UseAllTwocells

\input diagxy

\newdir{ >}{{}*!/-5pt/@{>}}

\usepackage{amsthm}
\usepackage{amsbsy}
\usepackage[colorlinks,urlcolor=black,linkcolor=black,citecolor=black]{hyperref}

\usepackage{ragged2e}
\usepackage{array}
\usepackage{enumerate}

\title{On the secondary Steenrod algebra} 

\author{Christian Nassau}  
\email{nassau@nullhomotopie.de, \url{http://nullhomotopie.de}}   % You may include the URL for your home page along with your email


\newcommand{\grade}[1]{\vert{#1}\vert}

\newcommand{\myref}[1]{(\ref{#1})}

\newcommand{\FF}{{\mathbb F}}
\newcommand{\ZZ}{{\mathbb Z}}
\newcommand{\GG}{{\mathbb G}}

\newcommand{\id}{{\rm id}}
\newcommand{\Sq}{{\rm Sq}}

\newcommand{\us}{_\ast}
\newcommand{\os}{^\ast}

\DeclareMathOperator{\cont}{\daleth}

\newcommand{\EBP}{{\mathrm{EBP}}}
\newcommand{\BP}{{\mathrm{BP}}}

\newcommand{\tef}X %{\text{\Taurus}}
\newcommand{\krypto}Y %{\text{\Pisces}}

\newcommand{\nproj}{\rho}
\newcommand{\gppol}{\Phi}

\newcommand{\cmap}{{\mathfrak c}}
\newcommand{\pmap}{{\mathfrak p}}

\newcommand{\catAlg}[1]{{{\rm Alg}^{\rm c}_{{#1}}}}

\DeclareMathOperator{\fotimes}{\hat\otimes}
\DeclareMathOperator{\totimes}{\widetilde{\otimes}}
\DeclareMathOperator{\cfotimes}{\widehat{\otimes}}

\DeclareMathOperator{\op}{op}

\DeclareMathOperator{\Hom}{Hom}

\DeclareMathOperator{\Sym}{Sym}

\DeclareMathOperator{\im}{im}

\DeclareMathOperator{\coker}{coker}

\DeclareMathOperator{\funcgraph}{graph}

\newcommand{\coz}{{\mathfrak X}} %{{\gimel}}
    
\theoremstyle{plain}
\newtheorem{theorem}{Theorem}

\newtheorem{thm}[theorem]{Theorem}
\newtheorem{lemma}[theorem]{Lemma}

\newtheorem{coro}[theorem]{Corollary}
\theoremstyle{remark}
\newtheorem{remark}[theorem]{Remark}

\RCS $Date: 2010/11/10 19:49:35 $
\RCS $Header: /jehova/safer/cn/cvsroot/misc/secmult/nss.tex,v 1.123 2010/11/10 19:49:35 cn Exp $
\RCS $Revision: 1.123 $

\numberwithin{equation}{section} 
\numberwithin{theorem}{section}

\begin{document} 
 
\begin{abstract} 
  We introduce a new model for the secondary Steenrod algebra
  at the prime $2$
  which is both smaller and more accessible than the 
  original construction of H.-J.\ Baues.
  
  We also explain how $\BP$ can be used to 
  define a variant of
  the secondary Steenrod algebra at odd primes.
%\\ \\ {\Huge $$\text{DRAFT VERSION \RCSRevision}$$}
      % Abstracts are required.
       % Please place the text of your abstract here.
\end{abstract} 
\date\today
\maketitle
\tableofcontents

%%%% **** The text of the paper starts here **** %%%%

  \begin{section}{Introduction}
    Let $A$ be the Steenrod algebra.
    In \cite{baues}, H.-J.\ Baues has constructed an exact sequence $B_\bullet$
    \begin{align}\label{bausec}
      \xymatrix{
        A \ar@{ >->}[r] & B_1 \ar[r]^-{\partial} & B_0 \ar@{ ->>}[r] & A
      }
    \end{align}
    which captures the algebraic structure of secondary cohomology 
    operations in ordinary mod $p$ cohomology. 
    This sequence is called the {\em secondary Steenrod algebra}
    and its knowledge allows, among other things, to give a purely algebraic
    description of the $d_2$-differential in the classical Adams spectral sequence
    (see \cite{MR2193337}).

    Unfortunately, the construction of $B_\bullet$ is not very explicit
    and apparently not many topologists have become familiar with it.
    The aim of the present note is to show that there is
    a smaller and much more accessible 
    model which captures the same information.
    In fact our model is so simple that we can describe it in this introduction:

    Fix $p=2$ and 
    let $D_0$ be the Hopf algebra that represents power series
    \begin{align*}%\label{powdef4}
      f(x) = \sum_{k\ge 0} \xi_kx^{2^k} + \sum_{0\le k<l} 2\xi_{k,l}x^{2^k+2^l}
    \end{align*}
    under composition modulo $4$.
    There is a natural map $\pi:D_0\twoheadrightarrow A$ and a decomposition
    \begin{align}\label{d0decomp}
      D_0 &= \ZZ/4\{\Sq(R)\} \oplus \sum_{-1\le k<l} Y_{k,l}A
    \end{align}
    where $\Sq(R), Y_{k,l}\in D_0$ are dual to $\xi^R$ resp.\ $\xi_{k+1,l+1}$ 
    with respect to the natural basis $\{\, \xi^R,\, 2\xi^R\xi_{k,l}\,\}$ of 
    ${D_0}\us=\ZZ/4[\xi_n,2\xi_{k,l}]$.

    Here are some computations that can help to become familiar with $D_0$: 
    $\Sq^1\Sq^1 = 2\Sq^2+ Y_{-1,0}$, $\Sq^1Y_{-1,0}=Y_{-1,0}\Sq^1+2\Sq(0,1)$.
    Let $Q_k = \Sq(\Delta_{k+1})$ for the exponent sequence $\Delta_k$ 
    with $\xi^{\Delta_k}=\xi_k$
    and $P_t^s = \Sq(2^s\Delta_t)$.
    Then $Q_0Q_k = \Sq(\Delta_1+\Delta_{k+1}) + Y_{-1,k}$ and $[Q_0,Q_k]=Y_{-1,k}$ if $k>0$.
    One also finds
    \begin{align*}
      P_t^s P_t^s &= \begin{cases}
        2 P_t^{s+1} & \text{\qquad($s+1<t$),} \\
        2 P_t^{s+1} + Y_{t-2,2s} \Sq\left((2^s-1)\Delta_{t}\right) & \text{\qquad($s+1 = t$).}
      \end{cases}
    \end{align*}
    So for example $\Sq(0,2) \cdot \Sq(0,2)=2 \Sq(0,4) + Y_{0,2} \Sq(0,1)$.
    More computations can be found in Figure \ref{ademtab}.    

    For products involving $Y_{k,l}$ there is the simple formula
    \begin{align}\label{yklcomm}
      a Y_{k,l} &= \sum_{i,j\ge 0} Y_{k+i,l+j} \cont(\xi_i^{2^{k+1}}\xi_j^{2^{l+1}},a)
      \intertext{if we interpret the $Y_{k,l}$ with $k\ge l$ as}
      \label{yklrel}
      Y_{k,l} &= \begin{cases} 
        Y_{l,k} & \text{\qquad($l<k$),} \\
        2\Sq(\Delta_{k+2})  & \text{\qquad($l=k$).}
      \end{cases}
    \end{align}
    Here we have written $\cont(p,a)$ for the contraction of $a\in A$ by $p\in A\us$
    defined via $\langle \cont(p,a),q\rangle = \langle a,pq\rangle$ for $q\in A\us$.
    Let 
    %$\kappa:A\rightarrow A$ be
    %Kristensen's 
    $\kappa(a)=\cont(\xi_1,a)$. % be Kristensen's derivation.

    %(A+\mu_0A + \sum_{\substack{-1\le k,l\\(k,l)\not=(-1,-1)}} U_{k,l} A)/\mathord{\sim}
    We now define our model $D_\bullet$ for the secondary Steenrod algebra 
    to be the sequence
    \begin{align*}
      \xymatrix@R=-0.5em{
        A \ar@{ >->}[r] & 
        {(A+\mu_0A + \,\,\raisebox{0pt}[0pt][0pt]{$\sum\nolimits_{-1\le k,\,0\le l}$}\,\, U_{k,l} A)/\mathord{\sim}}
        \ar[r]^-{\partial} & D_0 \ar@{ ->>}[r]^-{\pi} & A. \\
        & {\underbrace{{\hphantom{{(A+\mu_0A + \,\,{\sum\nolimits_{-1\le k,\,0\le l}}\,\, U_{k,l} A)/\mathord{\sim}}}}}_{=:D_1}}
      }
    \end{align*}  
    $D_1$ is an $A$-bimodule via $a\mu_0 = \mu_0a + \kappa(a)$
    and 
    \begin{align}\label{ucomm}
      a U_{k,l} &= \sum_{i,j\ge 0} U_{k+i,l+j} \cont(\xi_i^{2^{k+1}}\xi_j^{2^{l+1}},a).
    \end{align}    
    The relations defining $D_1$ are
    \begin{align}\label{urels}
      U_{k,l} &= \begin{cases} 
        U_{l,k} + \Sq(\Delta_{k+1}+\Delta_{l+1}) & \text{\qquad($l<k$),} \\
        \mu_0\Sq(\Delta_{k+2}) + \Sq(2\Delta_{k+1}) & \text{\qquad($l=k$).}
      \end{cases}
    \end{align}
    $\partial$ is zero on $A\subset D_1$ and 
    otherwise given by $\partial \mu_0a = 2a$ and $\partial U_{k,l}a = Y_{k,l}a$.

    The following is our main result:
    \begin{thm}\label{mainthm}
      There is a weak equivalence $B_\bullet\rightarrow D_\bullet$
      of crossed algebras that is the identity on $\pi_0$ and $\pi_1$.
    \end{thm}
    Recall that a crossed algebra \cite[5.1.6]{baues} 
    is an exact sequence of the form $B_\bullet$ with $B_0$ an algebra,
    $B_1$ a $B_0$-bimodule and a bilinear differential $\partial:B_1\rightarrow B_0$ 
    with $(\partial b)b' = b(\partial b')$ for $b,b'\in B_1$. 
    The homotopy groups
    $\pi_0(B_\bullet) := \coker\partial$ and $\pi_1(B_\bullet) := \ker\partial$
    will mostly be $A$ in our examples.

    This theorem makes it easy to compute threefold Massey products
    in the Steenrod algebra. 
    Think of $D_\bullet$ as the splice of the two short exact sequences
    \begin{align*}
      \xymatrix{
        A \ar@{ >->}[r] & D_1 \ar@{ ->>}[r]_-{\partial} & R_D\ar@{->}@/_1em/[l]_-{u}, 
        &\quad
        R_D \ar@{ >->}[r] & D_0 \ar@{ ->>}[r] & A \ar@{->}@/_1em/[l]_-{\sigma}        
      }
    \end{align*}
    and pick sections $\sigma$ and $u$ as indicated.
    A simple choice, for example, would be
    $\sigma(\sum c_i\Sq(R_i))$ $=$ $\sum \widehat{c_i}\Sq(R_i)$
    with $\widehat{(-)}:\ZZ/2\rightarrow\ZZ/4$ 
    given by $\widehat{0}=0$ and $\widehat{1}=1$.
    For $u$ one can take the map
    \begin{align}\label{udef}
      2\Sq(R) \mapsto \mu_0\Sq(R),\qquad
      Y_{k,l}\Sq(R)  \mapsto U_{k,l}\Sq(R)\qquad\text{(for $k<l$)}
    \end{align}
    which is right-linear.
    For $a,b\in A$ one then has $\sigma(ab)=\sigma(a)\sigma(b)+\partial\tau(a,b)$
    with $\tau(a,b) = u\left(\sigma(ab)-\sigma(a)\sigma(b)\right)\in D_1$.
    Associativity of the multiplication in $A$ dictates that
    \begin{align*}
      \langle a,b,c\rangle &:=
      \tau(ab,c) -\tau(a,b)\sigma(c) - \tau(a,bc) + \sigma(a)\tau(b,c)
    \end{align*}
    is a $\partial$-cycle, hence in $A$. $\langle a,b,c\rangle$
    is the Massey product in question. 
    It is only defined up to an indeterminacy 
    coming from the choices of $\sigma$ and $u$.

    As an example, consider the case $a=b=c=\Sq(0,2)$.
    With $\sigma$ and $u$ chosen as above
    one has $\sigma(a)\sigma(b)=2\Sq(0,4)+Y_{0,2}\Sq(0,1)$,
    so $\tau(a,b) = \mu_0\Sq(0,4) + U_{0,2}\Sq(0,1)$. One finds
    \begin{align*}
      \langle a,b,c\rangle &= \Sq(0,2)\tau(b,c) - \tau(a,b)\Sq(0,2) \\
      &= \mu_0\underbrace{\left[\Sq(0,2),\Sq(0,4)\right]}_{=\Sq(0,1,0,1)}
      + U_{0,2}\underbrace{\left[\Sq(0,2),\Sq(0,1)\right]}_{=0} + U_{2,2}\Sq(0,1) \\
      &= \mu_0\Sq(0,1,0,1) + \left(\mu_0\Sq(0,0,0,1)+\Sq(0,0,2)\right)\Sq(0,1) \\
      &= \Sq(0,1,2)
    \end{align*}
    which recovers a result of Kristensen and Madsen \cite{krimad}.
    A straightforward computation, whose details we leave to the interested reader,
    now generalizes this to
    \begin{coro}
      Let $t\ge 1$. Then $\langle P_t^s, P_t^s, P_t^s \rangle$
      is zero for $s<t-1$ and
      $\langle P_t^{t-1}, P_t^{t-1}, P_t^{t-1} \rangle 
      \owns \Sq\left((2^{t-1}-1)\Delta_t + 2^t\Delta_{t+1}\right)$.
    \end{coro}

    The plan of the paper is as follows.
    In the first section we will review the definition and structure 
    of $D_\bullet$ and sketch proofs for the claims in this introduction.
    In section \ref{edefsect} we will construct
    an intermediate sequence $E_\bullet$ with a weak equivalence
    $E_\bullet\rightarrow D_\bullet$.
    We then construct a comparison map $B_\bullet\rightarrow E_\bullet$
    in section \ref{compsect}, thereby proving the main Theorem. 
    Finally, the appendix sketches the relation of 
    the odd-primary secondary Steenrod algebra 
    with the algebra of $\BP$ operations.

    Before we proceed, however, I want to thank Mamuka Jibladze
    for many stimulating emails on the subject. The first such email
    arrived in May 2004 and this is when my interest in the secondary
    Steenrod algebra began. Without his guidance it would have been
    a lot more difficult to wrap my head around Baues's wonderful
    construction. 
    I also thank Hans-Joachim Baues for very constructive comments 
    on an earlier draft of this paper.
  \end{section}

  \begin{section}{The construction of \texorpdfstring{$D_\bullet$}{D}}
    \begin{subsection}{Definition}
      As in the introduction, we let
    \begin{align*}
      {D_0}\us &= \ZZ/4[\xi_k,2\xi_{k,l}\,\vert\,0\le k<l,\, \xi_0=1].
    \end{align*}
    This is turned into a Hopf algebra with coproduct
    \begin{align*}
      \Delta\left(\xi_n\right) &= 
      \sum_{i+j=n} \xi_i^{2^j}\otimes\xi_j 
      + 2 \sum_{0\le k<l}
      \xi_{n-1-k}^{2^k}\xi_{n-1-l}^{2^l}\otimes\xi_{k,l} \\
      \Delta\left(\xi_{n,m}\right) &= \xi_{n,m}\otimes 1 +
      \sum_{k\ge 0} \xi_{n-k}^{2^k}\xi_{m-k}^{2^k}\otimes\xi_{k+1}\\
      &\qquad+ \sum_{0\le k<l}\!\!
      \left(\xi_{n-k}^{2^k}\xi_{m-l}^{2^l}
      +\xi_{m-k}^{2^k}\xi_{n-l}^{2^l} \right)\!\otimes\xi_{k,l}. 
    \end{align*}
    We list some basic properties of its dual in the following
    \begin{lemma}\label{d0props}
      Let $D_0=\Hom({D_0}\us,\ZZ/4)$ be the dual algebra and
      let $\Sq(R)$, $Y_{k,l}(R)$ $\in D_0$ be defined by
      \begin{align*}
        \langle \Sq(R), \xi^S\rangle &= \delta_{R,S}, &
        \langle \Sq(R), 2\xi_{m,n}\xi^S\rangle &=0,\\
        \langle Y_{k,l}(R), \xi^S\rangle &= 0, &
        \langle Y_{k,l}(R), 2\xi_{m,n}\xi^S\rangle 
        &= 2\delta_{k+1,m}\delta_{l+1,n}\delta_{R,S}.
      \end{align*}
      Write $Y_{k,l}$ for $Y_{k,l}(0)$. The following is true:
      \begin{enumerate}
      \item
        There is a multiplicative map $\pi:D_0\rightarrow A$ 
        with $\Sq(R)\mapsto\Sq(R)$. 
      \item 
        One has $Y_{k,l}(R) = Y_{k,l}\Sq(R)$.
      \item
        %This is dual to the inclusion $A\us\cong 2{D_0}\us\hookrightarrow {D_0}\us$.
        The kernel $R_D=\ker\pi$ is $2D_0+\sum_{-1\le k<l} Y_{k,l}A$
        and satisfies $R_D^2=0$.
      \item
        The commutation rule \myref{yklcomm} holds
        with $Y_{k,l}$ as in \myref{yklrel} for $k\ge l$.
      \end{enumerate}
    \end{lemma}
    \begin{proof}
      The verification is straightforward.
    \end{proof}

    We will encounter the following $A$-bimodules more than once.
    \begin{lemma}
      There are $A$-bimodules $U$, $V$ with
      \begin{align*}
        V &= \sum_{-1\le k} V_kA, \quad
        U = \sum_{-1\le k,l} U_{k,l}A 
        \intertext{and relations}
        aV_k &= \sum_{i\ge 0}V_{k+i}\cont(\xi_i^{2^{k+1}},a), \quad
        aU_{k,l} = \sum_{i,j\ge 0}U_{k+i,l+j}\cont(\xi_i^{2^{k+1}}\xi_j^{2^{l+1}},a). 
        \intertext{Furthermore, let $R_{k,l}=U_{k,l}+U_{l,k}$
          and $R_{k,k}=U_{k,k}$  for $-1\le k<l$ and}
        K&=\sum_{-1\le k<l} R_{k,l}A + \sum_{-1\le k}R_{k,k}A.
      \end{align*}
      Then 
      \begin{align}\label{rcomrel1}
        aR_{k,l} &= \sum_{-1\le n<m} R_{n,m}
        \cont(\xi_{n-k}^{2^{k+1}}\xi_{m-l}^{2^{l+1}}+\xi_{m-k}^{2^{k+1}}\xi_{n-l}^{2^{l+1}},a), \\
        \label{rcomrel2}
        aR_{k,k} &= \sum_{0\le i} R_{k+i,k+i}\cont(\xi_i^{2^{k+2}},a) 
        + \sum_{0\le i<j} R_{k+i,l+j}\cont(\xi_i^{2^{k+1}}\xi_j^{2^{l+1}},a)
      \end{align}
      and $K$ is a bimodule, too.
      All of $U$, $V$ and $K$ are free $A$-modules from both
      left and right with basis the $U_{k,l}$, $V_k$, resp.\ $R_{k,l}$ and $R_{k,k}$.
      The same is true for the sub-bimodules 
      \begin{align*}
        V'&=\sum_{0\le k} V_kA,
        & U'&=\,\,\sum_{\makebox[0\width]{\scriptsize $-1\le k,\,0\le l$}}\,\, U_{k,l}A,
        & K'&=\sum_{\makebox[0\width]{\scriptsize $0\le k<l$}} R_{k,l}A + \sum_{0\le k}R_{k,k}A
      \end{align*}
      where the generators 
      $V_{-1}$, $U_{\ast,-1}$ and $R_{-1,\ast}$ have been left out.
    \end{lemma}
    \begin{proof}
      This is also straightforward.
    \end{proof}

    We will need the following computation in $A$.
    \begin{lemma}\label{qpcommrels}
      Let $a\in A$ and $k\ge 0$, $l\ge 1$. Then
      \begin{align}
        \label{qcomm}
        a Q_k &= \sum_{i\ge 0}Q_{k+i}\cont\left(\xi_i^{2^{k+1}},a\right), \\
        \label{pcomm}
        a P_l^1 &= \sum_{i\ge 0} P_{l+i}^1\cont\left(\xi_i^{2^{l+1}},a\right)
        + \kappa(a) Q_{l+1}  \\
        \nonumber
        &\qquad\qquad
        + \sum_{l\le i<j} Q_{i}Q_{j}\cont\left(\xi_{l-i}^{2^{l}}\xi_{l-j}^{2^{l}},a\right).
      \end{align}
    \end{lemma} 
    \begin{proof}
      Recall that $A\us$ is canonically an $A$-bimodule with
      $$\Delta(p) = \sum_R \Sq(R)p\otimes \xi^R = \sum_R \xi^R\otimes p\Sq(R).$$
      One has $\langle a\Sq(R),p\rangle = \langle a,\Sq(R)p\rangle$
      and $\langle \Sq(R)a,p\rangle = \langle a,p\Sq(R)\rangle$.
      Upon dualization \myref{qcomm} therefore becomes the identity
      $$Q_k p = \sum_{i\ge 0} (pQ_{k+i})\cdot \xi_i^{2^{k+1}}.$$
      Here both sides are derivations in $p$, so it only remains
      to check equality on the $\xi_n$ which is easily done.

      The second claim can be proved similarly, but with messier details.
      We leave this to the skeptical reader.
    \end{proof}

    The following Lemma is the key to the definition of $D_1$.
    Recall that $A+\mu_0A$ carries the bimodule structure 
    $a\mu_0=\mu_0a + \kappa(a)$.
    \begin{lemma}\label{lambdadef}
      There is a bilinear map $\lambda:K'\rightarrow A+\mu_0A$ with
      \begin{align*}
        R_{k,l} &\mapsto \Sq(\Delta_{k+1}+\Delta_{l+1}), \\
        R_{k,k} &\mapsto \Sq(2\Delta_{k+1}) + \mu_0\Sq(\Delta_{k+2}).
      \end{align*}
    \end{lemma}
    \begin{proof}
      We need to show that 
      %the $\lambda(R_{k,l})$ and $\lambda(R_{k,k})$ satisfy the same commutation
      $\lambda$ respects the
      relations \myref{rcomrel1} and  \myref{rcomrel2}.

      By \myref{qcomm} one has
      \begin{align*}
        aQ_kQ_l &= \sum_{i,j\ge 0} Q_{k+i}Q_{l+j}\cont(\xi_i^{2^{k+1}}\xi_j^{2^{l+1}},a).
      \end{align*}
      Using $Q_kQ_l=Q_lQ_k$ and $Q_k^2=0$ this immediately implies 
      compatibility with \myref{rcomrel1}.

      For \myref{rcomrel2} note 
      $a\lambda(R_{k,k}) = aP_{k+1}^1 + \kappa(a)Q_{k+1} + \mu_0 aQ_{k+1}$.
      The claim is therefore equivalent to
      \begin{align*}
        aP_{k+1}^1 + \kappa(a)Q_{k+1} &= 
        \sum_{0\le i} P_{k+i+1}^1\!\!\cont(\xi_i^{2^{k+2}},a) 
        + \!\sum_{0\le i<j} Q_{k+i}Q_{l+j}\!\!\cont(\xi_i^{2^{k+1}}\xi_j^{2^{l+1}},a), \\
        aQ_{k+1} &= \sum_{0\le i} Q_{k+i+1}\cont(\xi_i^{2^{k+2}},a). 
      \end{align*}
      These are again just variants of \myref{qcomm} and \myref{pcomm}.
    \end{proof}

    Now let $D_1 = (A+\mu_0A + U')/\funcgraph(\lambda)$.
    This is easily seen to agree
    with the definition in the introduction.
    \begin{lemma}
      Let  $\partial U_{k,l} = Y_{k,l}$ and $\partial\mu_0=2$.
      This defines an exact sequence   
      \begin{align*}
        \xymatrix@R=-0.5em{
          A \ar@{ >->}[r] & D_1 
          \ar[r]^-{\partial} & D_0 \ar@{ ->>}[r]^-{\pi} & A.
        }
      \end{align*} 
    \end{lemma}
    \begin{proof}
      Lemma \ref{lambdadef} shows that $D_1$ is indeed a bimodule. 
      That $\partial$ is well-defined and bilinear
      follows from the relations \myref{yklcomm}.
      Finally, $D_1$ can be written as the direct sum 
      $$D_1=A+\mu_0A+\sum_{-1\le k<l} U_{k,l}A.$$
      From this the exactness of the sequence is obvious.
    \end{proof}

    \end{subsection}
    \begin{subsection}{Represented Functors}
      Some of the previous constructions can be given meaningful descriptions
      when we look at their associated functors. Unfortunately, we have not been
      able to find a good explication for the map $\lambda$, so we eventually have to
      resort to pure algebra in our construction of $D_\bullet$.

      Let $\catAlg{\ZZ/4}$ be the category of commutative algebras over $\ZZ/4$.
      \begin{lemma}\label{d0modular}
        There is a natural isomorphism
        $\Hom_{\catAlg{\ZZ/4}}\left({D_0}\us,-\right) \overset\cong{\longrightarrow} G(-)$
        where $G(R)$ $\subset$ $R[[x]]$ is the group
        \begin{align*}
          \Big\{\, f(x) = \sum_{k\ge 0} t_kx^{2^k} + \sum_{0\le k<l} t_{k,l}x^{2^k+2^l}
          \,\Big\vert\,
          \text{$t_0=1$, $J^2=0$ for $J=(2,t_{k,l})\subset R$}\,\Big\}.
        \end{align*}
      \end{lemma}
      \begin{proof}
        A $\phi:{D_0}\us\rightarrow R$ maps to the $f$ with $t_k=\phi(\xi_k)$
        and $t_{k,l}=\phi(2\xi_{k,l})$.
      \end{proof}
      
      The bimodules $U$ and $V$ can be understood by looking at the 
      functors
      \begin{align*}
        V_!(R) &= G(R) \times \Big\{ v(x) = \sum_{k\ge 0} v_{k}x^{2^k} 
        \,\Big\vert\,
        v(x)^2 = 2v(x) = 0
        %\text{$J^2=0$ for $J=(2,v_{k})\subset R$}
        \,\Big\}, \\
        U_!(R) &= G(R) \times \Big\{ f_2(x,y) = \sum_{k,l\ge 0} u_{k,l}x^{2^k}y^{2^l} 
        \,\Big\vert\,
        f_2(x,y)^2 = 2f_2(x,y) = 0
        %\text{$J^2=0$ for $J=(2,u_{k,l})\subset R$}
        \,\Big\}.
      \end{align*}
      The group operation is given by 
      $(f_1,v)$ $\circ$ $(g_1,w)$ $=$ $(f_1g_1,vg_1+w)$ resp.\ 
      $(f_1,f_2)$ $\circ$ $(g_1,g_2)$ $=$ $(f_1g_1,f_2(g_1\times g_1)+g_2)$.
      
      $V_!$ and $U_!$ are represented by algebras 
      ${D_0}\us[v_k]/J^2$ and ${D_0}\us[u_{k,l}]/J^2$
      where $J$ is the ideal $(2,v_k)$ resp.\ $(2,u_{k,l})$.
      $V$ and $U$ can then be recovered as the duals of the degree $1$ part of these
      algebras.

      We can use this to at least partially explain the map from $U$ to $D_0$.
      \begin{lemma}
        The map $\phi:U\rightarrow D_0$ with $U_{k,l}\mapsto Y_{k,l}$ 
        and $U_{k,k}\mapsto 2Q_{k+1}$ is associated to the natural transformation
        $$U(R)\owns f=(f_1,f_2) \mapsto f^{\rm eff}\in G(R)$$
        with $f^{\rm eff}(x) = f_1(x) + f_2(x,x)$.
      \end{lemma}
      \begin{proof}        
        We have an isomorphism ${D_0}\us[u_{k,l}]/J^2 = {D_0}\us[2w_{k,l}]$
        and will use the $w_{k,l}$ in our computation for the sake of clarity.
        Recall that 
        $\langle Q_ka,p\rangle = \langle a,(\partial p)/(\partial \xi_{k+1})\rangle$
        for $a\in A$, $p\in A\us$. 
        Therefore the dual $\phi\us: {D_0}\us\rightarrow U\us$
        is given by 
        $$p \mapsto 2\sum_{k\ge 0} (\partial p)/(\partial \xi_{k+1}) w_{k,k}
        + \sum_{0\le k<l} 2(\partial p)/(\partial \xi_{k,l}) \left(w_{k,l}+w_{l,k}\right).
        $$
        The map $\widehat{\phi\us}:{D_0}\us\rightarrow {D_0}\us[2w_{k,l}]$ 
        with
        $p\mapsto p + \phi\us(p)$ is multiplicative since $\phi\us$ is a derivation.
        It therefore does correspond to a 
        natural transformation $U_!(R)\rightarrow G(R)$.
        To see that this transformation is $f\mapsto f^{\rm eff}$ 
        one just has to check that
        $\widehat{\phi\us}(\xi_{n+1}) = \xi_{n+1} + 2w_{n,n}$ and
        $\widehat{\phi\us}(2\xi_{k,l}) = 2\xi_{k,l} + 2w_{k,l} + 2w_{l,k}$.
      \end{proof}
      The bilinearity of $\phi$ expresses the fact, that
      $f\mapsto f^{\rm eff}$ is multiplicative. 
      This is also easy to see computationally.
      \begin{lemma}
        One has $(fg)^{\rm eff} = f^{\rm eff} \circ g^{\rm eff}$.
      \end{lemma}
      \begin{proof}
        We have
        \begin{align*}
          (fg)^{\rm eff}(x) &= f_1(g_1(x)) + f_2(g_1(x),g_1(x)) + g_2(x,x), \\
          f^{\rm eff}(g^{\rm eff}(x)) &=
          f_1(g_1(x)+g_2(x,x)) + f_2(g_1(x)+g_2(x,x),g_1(x)+g_2(x,x)).
        \end{align*}
        Since  $g_2^k = 0$ for $k\ge 2$ 
        we have
        \begin{align*}
          f_1(g_1(x)+g_2(x,x)) &= f_1(g_1(x))+g_2(x,x),\\
          f_2(g_1(x)+g_2(x,x),g_1(x)+g_2(x,x)) &= f_2(g_1(x),g_1(x))
        \end{align*}
        which implies $(fg)^{\rm eff}(x) = f^{\rm eff}(g^{\rm eff}(x))$.
      \end{proof}
      
    %Note that this Lemma implies the relations \myref{yklcomm} and \myref{yklrel}.

    \end{subsection}
  \end{section}

  %%%%%%%%%%%%%%%%%%%%%%%%%%%%%%%%%%%%%%%%%%%%%%%%%
  %%%%%%%%%%%%%%%%%%%%%%%%%%%%%%%%%%%%%%%%%%%%%%%%%
  %%%%%%%%%%%%%%%%%%%%%%%%%%%%%%%%%%%%%%%%%%%%%%%%%
  %%%%%%%%%%%%%%%%%%%%%%%%%%%%%%%%%%%%%%%%%%%%%%%%%
  %%%%%%%%%%%%%%%%%%%%%%%%%%%%%%%%%%%%%%%%%%%%%%%%%

  \begin{section}{The construction of \texorpdfstring{$E_\bullet$}{E}}\label{edefsect}
    We now prepare ourselves for the comparison between
    our $D_\bullet$ and the $B_\bullet$ of Baues.
    It turns out that an intermediate $E_\bullet$ is required.
    The reason is that $D_\bullet$, although 
    sufficient for the computational
    applications of the theory, does not capture all of the 
    structure of $B_\bullet$. The latter carries a comultiplication
    which turns it into a {\em secondary Hopf algebra} and
    the associated invariants $L$ and $S$ are crucial for the 
    comparison.
    We will therefore now pass to a slightly larger $E_\bullet$
    where this extra structure can be expressed.
    \begin{subsection}{Definition}
      Let $X = \sum_{-1\le k,l} X_{k,l}A$ be a copy of $U$ with $U_{k,l}$
      renamed $X_{k,l}$ and let 
      $X'\subset X$ be the subspace without $X_{-1,-1}A$.
      Let $\widehat{E_k} = D_k + X' + \mu_0X'$ for $k=0,1$.
      We will write $e = e_D + e_X$ for the decomposition of $e\in \widehat{E_k}$
      into the $D_k$ and $X + \mu_0X$ components.
      Let $\nproj:E_\bullet\rightarrow D_\bullet$ denote the projection $e\mapsto e_D$.
      We extend $\partial$ to $\widehat{E_\bullet}$ via  
      $\partial e = \partial e_D + e_X$.
      This defines an exact sequence
      \begin{align}\label{etildedef}
        \xymatrix@1@R=-0.5em{ A \ar@{ >->}[r] &
          \raisebox{0pt}[0pt][0pt]{$\widehat{E_1}$} \ar[r]^-{\partial} &
          \raisebox{0pt}[0pt][0pt]{$\widehat{E_0}$} \ar@{ ->>}[r]^-{\pi} & A.  }
      \end{align} 
      We need to define a multiplication on $\widehat{E_0}$.
      Note that there is an isomorphism 
      $U\cong V\otimes_A V$ where $U_{k,l}\leftrightarrow V_k\otimes V_l$.
      We can therefore write $X_{k,l}=X_kX_l$ where the 
      $X_k$ are generators of a copy $V_X$ of $V$. 
      Let $\psi:A\rightarrow V_X'$ be given by 
      $\psi(a) = \sum_{k\ge 0} X_k\cont(\xi_{k+1},a)$.
      $\psi$ is a derivation because one has $\psi(a) = X_{-1}a-aX_{-1}$.
      Recall that  $\kappa:A\rightarrow A$  is also a derivation.
      \begin{lemma}
        Let $\ast:D_0\otimes D_0\rightarrow D_0 + X + \mu_0X$ be given by
        \begin{align}\label{muldef}
          a\ast b&=ab+\psi(a)\psi(b)\mu_0+X_{-1}\psi(a)\kappa(b)
        \end{align}
        and extend this to all of $\widehat{E_0}$ via
        $d\ast m = \pi(d)m$, $m\ast d = m\pi(d)$ and $mm'=0$ 
        for $d\in D_0$ and $m,m'\in X+\mu_0X$.
        Then $\ast$ is associative.
      \end{lemma}
      \begin{proof}
        The only questionable case is when all three factors are in $D_0$.
        But this is a straightforward computation:
        \begin{align*}
          &(a\ast b)\ast c = \\
          &=abc+ \psi(ab)\psi(c)\mu_0+X_{-1}\psi(ab)\kappa(c)
          +\psi(a)\psi(b)\mu_0c+X_{-1}\psi(a)\kappa(b)c \\
          &= abc+\psi(a)b\psi(c)\mu_0+a\psi(b)\psi(c)\mu_0
          +X_{-1}\psi(a)b\kappa(c)+X_{-1}a\psi(b)\kappa(c)\\
          &\qquad+\psi(a)\psi(b)c\mu_0+\psi(a)\psi(b)\kappa(c)
          +X_{-1}\psi(a)\kappa(b)c, \\
          &a\ast (b\ast c) = \\
          &=abc+\psi(a)\psi(bc)\mu_0+X_{-1}\psi(a)\kappa(bc) 
          + a\psi(b)\psi(c)\mu_0+aX_{-1}\psi(b)\kappa(c) \\
          &= abc+\psi(a)b\psi(c)\mu_0+\psi(a)\psi(b)c\mu_0
          +X_{-1}\psi(a)\kappa(b)c+X_{-1}\psi(a)b\kappa(c) \\
          &\qquad+a\psi(b)\psi(c)\mu_0
          +X_{-1}a\psi(b)\kappa(c)+\psi(a)\psi(b)\kappa(c).
        \end{align*}
      \end{proof}

      Figure \ref{ademtab}
      illustrates the multiplication in $E_0$
      with the computation of the first few Adem relations.

      We will define $E_0\subset \widehat{E_0}$ by a condition
      on the coefficients of $Y_{-1,\ast}$, $X_{-1,\ast}$ and $X_{\ast,-1}$.
      To formulate that condition we need to define two more maps.
      \begin{lemma}\label{psitheta}
        Let $\theta_D:D_0\rightarrow V$
        be the map that extracts the $Y_{-1,k}$.
        In other words, let
        $$\theta_D(\Sq(R))=0,\quad
        \theta_D(Y_{-1,n}a) = V_na,\quad
        \theta_D(Y_{k,l}a) = 0 \quad\text{for $k\not=-1$}.$$
        Then $\widehat{\theta_D}:D_0\rightarrow V+\mu_0V$ with
        $\widehat{\theta_D}(d) = \theta_D(d) + \psi(d)\mu_0$
        is a derivation.
      \end{lemma}
      \begin{proof}
        We sketch a quick computational proof here. A better argument will
        be given later from the functorial point of view.

        We already know that $\psi$ is a derivation, so we just need to show
        $\theta_D(de)=d\theta_D(e)+\theta_D(d)e + \psi(d)\kappa(e)$.
        Since $\theta_D$ sees only the $\xi_{0,n}$ 
        we can compute $\theta_D(de)$ from the coproduct formula
        $$\Delta \xi_{0,n} =  \xi_{0,n}\otimes 1 + \sum_{k\ge 0}
        \xi_{n-k}^{2^k}\otimes \xi_{0,k} + \xi_{n-1}\otimes\xi_1$$
        and these summands translate to
        $\theta_D(d)e$, $d\theta_D(e)$ and $\psi(d)\kappa(e)$.
      \end{proof}

      Similarly, let $\theta_E:\widehat{E_0}\rightarrow V$ 
      extract the $X_{-1,k}$:
      \begin{align*}
        &\theta_E(X_{-1,k}a) = V_ka,\quad
        \theta_E(X_{l,-1}a) = 0, \\
        &\theta_E\left(D_0+\mu_0X+\sum\nolimits_{k,l\ge 0} X_{k,l}A\right) = 0.\quad
      \end{align*}
      \begin{lemma}\label{e0defprep}
        One has 
        $\theta_E(d\ast e) = \theta_E(d)e+d\theta_E(e)+\psi(d_D)\kappa(e_D)$
        for $d,e\in \widehat{E_0}$.
      \end{lemma}
      \begin{proof}
        This is a straightforward computation.
        See also the discussion in Remark \ref{emultdiscuss} below.
      \end{proof}
 
      \begin{lemma}
        Define 
        \begin{align*}
          \widetilde{E_0} &= D_0 + 
          \sum_{k,l\ge0} X_{k,l}A 
          + \sum_{k,l\ge0} \mu_0X_{k,l}A 
          + \sum_{k\ge 0} X_{-1,k}A\subset \widehat{E_0}
        \end{align*}
        and let $E_0\subset\widetilde{E_0}$ be the subset
        where $\theta_D\circ\nproj$ and $\theta_E$ coincide.
        Then $E_0$ is closed under the multiplication $\ast$.
      \end{lemma}
      \begin{proof}
        It's clear that $\widetilde{E_0}$ is multiplicatively closed
        since $\ast$ cannot generate any $X_{k,-1}$ 
        if this is not already part of one factor.
        
        That $E_0$ is also multiplicatively closed
        follows from the identical formulas for $\theta_D(de)$
        and $\theta_E(de)$.
      \end{proof}
      
      \begin{coro}
        Let $E_1=\partial^{-1}(E_0)\subset \widehat{E_1}$. Then
        \begin{align}\label{edef}
          \xymatrix@R=-0.5em{
            A \ar@{ >->}[r] & E_1
            \ar[r]^-{\partial} & E_0 \ar@{ ->>}[r]^-{\pi} & A.
          }
        \end{align} 
        is a crossed algebra $E_\bullet$ with a canonical 
        projection $\nproj:E_\bullet\twoheadrightarrow D_\bullet$.
      \end{coro}
      \begin{proof}
        Clear.
      \end{proof}

      \begin{figure}
        \begin{RaggedRight}
          \newcommand{\xpl}{+}
          \begin{tabular}{c|p{1.8cm}|p{3.7cm}|p{4.7cm}}
            $[n,m]$ & Definition & $D_0$ & $X+\mu_0X$ \\
            \hline
            %
            %\\
            \hline
            $ [1,1] $ & $  1\cdot 1  $
            &  $ 2\,\Sq(2)  $ + $  Y_{-1,0} $ & $  X_{-1,0}  $ + $  \mu_0X_{0,0}  $ 
            \\
            \hline
            $ [1,2] $ & $  1\cdot 2  +  3  $
            &  $  Y_{-1,0} \Sq(1) $ & $  X_{-1,0} \Sq(1)  $ + $  \mu_0X_{0,0} \Sq(1)  $ + $ X_{0,0}  $ 
            \\
            \hline
            $ [2,2] $ & $  2\cdot 2  +  3\cdot 1  $
            &  $ 2\,\Sq(1,1) \xpl 2\,\Sq(4)  $ + $  Y_{-1,0} \Sq(2) $ & $  X_{-1,0} \Sq(2)  $ + $  X_{0,0} \Sq(1)  $ + $  \mu_0X_{0,0} \Sq(2)  $ + $  \mu_0X_{0,1}  $ 
            \\
            $ [1,3] $ & $  1\cdot 3  $
            &  $  Y_{-1,0} \Sq(2) $ & $  X_{-1,0} \Sq(2)  $ + $  \mu_0X_{0,0} \Sq(2)  $ + $ X_{0,0} \Sq(1)  $ 
            \\
            \hline
            $ [3,2] $ & $  3\cdot 2  $
            &  $ 2\,\Sq(2,1) \xpl 2\,\Sq(5)  $ + $  Y_{-1,0} ( \Sq(0,1) \xpl \Sq(3) ) $ & $  X_{-1,0} ( \Sq(0,1) \xpl \Sq(3) )  $ + $  X_{0,0} \Sq(2)  $ + $  X_{0,1}  $ + $  \mu_0X_{0,0} ( \Sq(0,1) \xpl \Sq(3) )  $ + $  \mu_0X_{0,1} \Sq(1)  $ 
            \\
            $ [2,3] $ & $  2\cdot 3  +  4\cdot 1  +  5  $
            &  $ 2\,\Sq(2,1)  $ & $  X_{0,1}  $ + $  \mu_0X_{0,1} \Sq(1) $ 
            \\
            $ [1,4] $ & $  1\cdot 4  +  5  $
            &  $ 2\,\Sq(5)  $ + $  Y_{-1,0} \Sq(3) $ & $  X_{-1,0} \Sq(3)  $ + $  X_{0,0} \Sq(2)  $ + $  \mu_0X_{0,0} \Sq(3)  $ 
            \\
            \hline
            $ [3,3] $ & $  3\cdot 3  +  5\cdot 1  $
            &  $ 2\,\Sq(6)  $ + $  Y_{-1,0} ( \Sq(1,1) \xpl \Sq(4) ) $ & $  X_{-1,0} ( \Sq(1,1) \xpl \Sq(4) )  $ + $  X_{0,0} ( \Sq(0,1) \xpl \Sq(3) )  $ + $  \mu_0X_{0,0} ( \Sq(1,1) \xpl \Sq(4) )  $ 
            \\
            $ [2,4] $ & $  2\cdot 4  +  5\cdot 1  +  6  $
            &  $ 2\,\Sq(3,1) \xpl 2\,\Sq(6)  $ + $  Y_{-1,0} \Sq(4) $ & $  X_{-1,0} \Sq(4)  $ + $  X_{0,0} \Sq(3)  $ + $  X_{0,1} \Sq(1)  $ + $  \mu_0X_{0,0} \Sq(4)  $ + $  \mu_0X_{0,1} \Sq(2)  $ 
            \\
            $ [1,5] $ & $  1\cdot 5  $
            &  $ 2\,\Sq(6)  $ + $  Y_{-1,0} \Sq(4) $ & $  X_{-1,0} \Sq(4)  $ + $  X_{0,0} \Sq(3)  $ + $  \mu_0X_{0,0} \Sq(4)  $ 
            \\
            \hline
            $ [4,3] $ & $  4\cdot 3  +  5\cdot 2  $
            &  $ 2\,\Sq(1,2) \xpl 2\,\Sq(4,1)  $ + $  Y_{-1,0} ( \Sq(2,1) \xpl \Sq(5) ) $ & $  X_{-1,0} ( \Sq(2,1) \xpl \Sq(5) )  $ + $  X_{0,0} ( \Sq(1,1) \xpl \Sq(4) )  $ + $  \mu_0X_{0,0} ( \Sq(2,1) \xpl \Sq(5) )  $ + $  \mu_0X_{0,1} \Sq(0,1)  $ 
            \\
            $ [3,4] $ & $  3\cdot 4  +  7  $
            &  $  Y_{-1,0} \Sq(2,1) $ & $  X_{-1,0} \Sq(2,1)  $ + $  X_{0,1} \Sq(2)  $ + $  \mu_0X_{0,0} \Sq(2,1)  $ + $  \mu_0X_{0,1} \Sq(3)  $ + $ X_{0,0} \Sq(1,1)  $ 
            \\
            $ [2,5] $ & $  2\cdot 5  +  6\cdot 1  $
            &  $ 2\,\Sq(4,1)  $ & $  X_{0,1} \Sq(2)  $ + $  \mu_0X_{0,1} \Sq(3) $ 
            \\
            $ [1,6] $ & $  1\cdot 6  +  7  $
            &  $  Y_{-1,0} \Sq(5) $ & $  X_{-1,0} \Sq(5)  $ + $  \mu_0X_{0,0} \Sq(5)  $ + $ X_{0,0} \Sq(4)  $ 
          \end{tabular}
        \end{RaggedRight}
        \caption{List of Adem relations in $E_0$.}
        \label{ademtab}
      \end{figure}
        
    \end{subsection}
    \begin{subsection}{Represented Functors}
      \begin{lemma}
        For $f(x)\in G(R)$ let $\tau_f(x)$ and $\theta_f(x)$ be defined by the
        decomposition
        \begin{align}
          \label{thetadef}
          f(x) &= x + \tau_f(x^2) + x\theta_f(x^2)
        \end{align}
        and write $\overline f(x)=f(x)-x$.
        Then
        \begin{align}
          \label{psicompform}
          \overline {fg}(x) &= \overline{f}(g(x)) + \overline{g}(x), \\
          \label{thetacompform}
          \theta_{fg}(x) &= \theta_f(g(x)) +\theta_g(x) + \xi_1^f\overline{g}(x),
        \end{align}
        where $\xi_1^f = \tau_f'(0)$ is the coefficient of $x^2$ in $f(x)$.
      \end{lemma}
      \begin{proof}
        This is a straightforward computation.
      \end{proof}
      Recall that $V$ represents the functor
      $$V_!(R) \cong G(R)\times\Big\{v(x)=\sum_{k\ge 1} v_kx^{2^k}
      \,\vert\, v(x)^2=0,\, 2v(x)=0\Big\}.$$
      This extends to $M=V+\mu_0 V$ as
      $$M_!(R) \cong G(R)\times 
      \Big\{v(x)=v_0(x)+\mu_0v_1(x)\,\vert\,\text{$v_0$, $v_1$ as in $V_!(R)$}\Big\}$$
      where 
      $$(f,v_0+\mu_0v_1)\circ (g,w_0+\mu_0w_1) 
      = (fg, v_0g+w_0+\xi_1^fw_1 + \mu_0(v_1g+w_1)).$$
      We can use this to give an explanation of $\psi$ and $\theta_D$.
      \begin{lemma}
        Let $\widehat{\theta_D}$ be the derivation 
        $D_0\rightarrow V+\mu_0V = M$ from Lemma \ref{psitheta}
        and let $\widetilde{\theta_D}:\Sym_{{D_0}\us}(M\us)\rightarrow {D_0}\us$
        be the multiplicative extension with 
        $\widetilde{\theta_D}\vert_{M\us} = \widehat{\theta_D}\us$.
        Then  $\widetilde{\theta_D}$ represents
        the transformation $G(R)\rightarrow M_!(R)$
        with $f\mapsto (f,\theta_f(x)+ \mu_0\overline{f}(x))$.
      \end{lemma}
      \begin{proof} 
        For an $f(x)$ of the form $\sum_{k\ge 0}x^{2^k}+\sum_{0\le k<l}2\xi_{k,l}x^{2^k+2^l}$
        one has
        \begin{align*}
          \tau_f(x) &= 
          \sum_{k\ge 1} \xi_kx^{2^{k-1}} + \sum_{1\le k<l}2\xi_{k,l}x^{2^{k-1}+2^{l-1}}, \\
          \theta_f(x) &= \sum_{k\ge 0} 2\xi_{0,k}x^{2^k}.
        \end{align*}
        The map $f\mapsto (f,\theta_f(x)+ \mu_0\overline{f}(x))$
        therefore corresponds to the $M\us\rightarrow {D_0}\us$
        with
        $v_k\mapsto 2\xi_{0,k}$ and $\mu_0\os v_k\mapsto \xi_{k}$.
        But this is just $\widehat{\theta_D}\us$.
      \end{proof}
      The multiplicative properties of $\psi$ and $\theta_D$ 
      that we established in Lemma \ref{psitheta}
      are therefore just a reformulation of \myref{psicompform} 
      and \myref{thetacompform}.

      We can now translate the definition of $E_0$ into
      the functorial context.
      \begin{lemma}
        The ring $\widehat{E_0}$
        represents pairs $(f_1(x),f_2(x,y))$ with $f_1(x)\in G(R)$
        and $f_2(x,y)=f_2^{\scriptscriptstyle{(0)}}(x,y)
        +\mu_0f_2^{\scriptscriptstyle{(1)}}(x,y)$ with 
        $(f_1,f_2^{\scriptscriptstyle{(j)}})\in U_!(R)$.
        The multiplication $\ast$ corresponds to the composition
        \begin{align*}
          \left(f\circ g\right)_2(x,y) &=
          f_2\left(g_1(x),g_1(y)\right) 
          + \xi_1^f\cdot g_2^{\scriptscriptstyle{(1)}}(x,y) + g_2(x,y) \\
          &\quad + \mu_0^f \overline{g}(x)\cdot \overline{f}(g(y))
          + \xi_1^f x \cdot \overline{g}(y).
        \end{align*}
        The subset of those $(f_1,f_2)$ with 
        $$f_2(x,y)=x\cdot\theta_{f_1}(y^2) + f_2^{\scriptscriptstyle{(0)}}(x^2,y^2) 
        + \mu_0f_2^{\scriptscriptstyle{(1)}}(x^2,y^2)$$
        is closed under $\ast$ and represented by $E_0$.
      \end{lemma}
      \begin{proof}
        Again this is straightforward.
      \end{proof}

      \begin{remark}\label{emultdiscuss}
        Rephrasing the previous discussion one could say that
        in $E_0$ we are studying certain pairs 
        $f = (f_1,f_2)$ under the transformation rule
        \begin{align*}
          (fg)_1 &= f_1g_1,\quad
          (fg)_2(x,y) = (fg)_2^{\rm basic}(x,y) + \text{correction terms}
          \intertext{where}
          (fg)_2^{\rm basic}(x,y) &= f_2\left(g_1(x),g_1(y)\right) 
          + \xi_1^f\cdot g_2^{\scriptscriptstyle{(1)}}(x,y) 
          + g_2(x,y).
        \end{align*}
        Here the correction terms are specifically crafted
        to preserve the conditions
        \begin{align*}
          %\label{ep1}
          f_2(x,y) &\equiv 0 \mod y^2, \\
          %\label{ep2}
          f_2(x,y) &\equiv x\theta_{f_1}(y^2) \mod x^2
        \end{align*}
        that define $E_0$.
        %(Note that in the terminology of Lemma \ref{e0defprep}
        %the first condition alone defines $\widetilde{E_0}$.)
        To us this suggests that the basic object of study should
        be the composition $(fg)_2^{\rm basic}$ and the subspace $E_0$,
        both of which have a reasonably elementary definition.
        The precise structure of the correction terms 
        might then count as an artifact of the retraction from $\widehat{E_0}$
        to $E_0$.
      \end{remark}

    \end{subsection}
  \end{section}

  %%%%%%%%%%%%%%%%%%%%%%%%%%%%%%%%%%%%%%%%%%%%%%%%%
  %%%%%%%%%%%%%%%%%%%%%%%%%%%%%%%%%%%%%%%%%%%%%%%%%
  %%%%%%%%%%%%%%%%%%%%%%%%%%%%%%%%%%%%%%%%%%%%%%%%%
  %%%%%%%%%%%%%%%%%%%%%%%%%%%%%%%%%%%%%%%%%%%%%%%%%
  %%%%%%%%%%%%%%%%%%%%%%%%%%%%%%%%%%%%%%%%%%%%%%%%%

\newcommand{\aug}{\epsilon}

  \begin{section}{The Hopf structure on \texorpdfstring{$E_\bullet$}{E}}\label{compsect}
    The secondary Steenrod algebra 
    %$B_\bullet$ 
    comes equipped
    with a diagonal $B_\bullet\rightarrow B_\bullet\fotimes B_\bullet$
    that extends the usual coproducts on $A$ and $B_0$. 
    This extra structure is essential for the characterization of $B_\bullet$
    in the Uniqueness Theorem \cite[15.3.13]{baues}.    
    In this section we are going to exhibit a similar structure on $E_\bullet$,
    which is a key step in our proof that $B_\bullet\sim E_\bullet$.

    \begin{subsection}{\texorpdfstring{$E_0$}{E0} as Hopf algebra}
      \begin{lemma}
        There is a unique multiplicative $\Delta_0:E_0\rightarrow E_0\otimes E_0$
        with $$\Delta_0\left(\Sq(R)\right) = \sum_{E+F=R} \Sq(E)\otimes \Sq(F)$$
        and $\Delta_0(Z) = Z\otimes 1 + 1\otimes Z$
        for $Z\in\{ Y_{k,l}, X_{k,l}, \mu_0X_{k,l} \}$.
      \end{lemma}
      \begin{proof}
        The uniqueness is clear. 
        To show existence, we begin with  the dual of the 
        multiplication map ${D_0}\us\otimes {D_0}\us \rightarrow {D_0}\us$.
        This defines a $\Delta_0:D_0\rightarrow D_0\otimes D_0$
        with $\Delta_0(Y_{k,l}) = Y_{k,l}\otimes 1 + 1\otimes Y_{k,l}$.
        We extend this to all of $E_0$
        via $\Delta_0(Z\cdot \Sq(R)) = (Z\otimes 1 + 1\otimes Z)\cdot \Delta(\Sq(R))$
        for $Z\in\{X_{k,l}, \mu_0X_{k,l} \}$. We have to show that this map
        is multiplicative.

        This is a straightforward computation, 
        and we will work out only one representative 
        case. Let $a\in A$ and $\Delta a = \sum a'\otimes a''$. Then
        \begin{align*}
          \Delta_0\left(aX_{k,l}\right) &= \Delta_0\Big(\sum\nolimits_{i,j\ge 0} 
          X_{k+i,l+j}\cont\left(\xi_i^{2^{k+1}}\xi_j^{2^{l+1}},a\right)\Big)
          \\&= \sum_{i,j\ge 0} 
          \left(X_{k+i,l+j}\otimes 1 + 1\otimes X_{k+i,l+j}\right)
          \Delta_0\left(\cont\left(\xi_i^{2^{k+1}}\xi_j^{2^{l+1}},a\right)\right) 
          \\&= \sum_{a', a''} 
          \sum_{i,j\ge 0} \left\{
          \left(X_{k+i,l+j}\cont\left(\xi_i^{2^{k+1}}\xi_j^{2^{l+1}},a'\right)\right)
          \otimes a'' \right.\\
          & \qquad \qquad\qquad \qquad \left.
          + \, a'\otimes \left(X_{k+i,l+j}\cont\left(\xi_i^{2^{k+1}}\xi_j^{2^{l+1}},a''\right)\right)
          \right\}
          \\&= \sum_{a', a''}  \left(a'X_{k,l}\otimes a'' + a'\otimes a''X_{k,l}\right)
        \end{align*}
        where we have used 
        $\Delta\cont(p,a) = \sum \cont(p,a')\otimes a'' = 
        \sum a'\otimes \cont(p,a'')$.
        This shows
        $\Delta_0(aX_{k,l})=\Delta_0(a)\Delta_0(X_{k,l})$.
        We
        leave the remaining cases to the reader.
      \end{proof}

      There is also a canonical augmentation $\aug:E_0\rightarrow \ZZ/4$
      which is dual to the inclusion $\ZZ/4\subset {D_0}\us\subset {E_0}\us$. 
      The following corollary is then obvious.
      \begin{coro}\label{e0hopf}
        $E_0$ is a Hopf algebra over $\ZZ/4$ with augmentation $\aug$ and 
        coproduct
        $\Delta_0$.
        The projection $E_0\rightarrow A$ is a map of Hopf algebras.
      \end{coro}
    \end{subsection}

    \begin{subsection}{The folding product}
      We next want to define a secondary diagonal
      $\Delta_1:E_1\rightarrow (E\fotimes E)_1$. 
      This requires a short discussion of the folding product
      $(E\fotimes E)_\bullet$ that figures on the right hand side.
      The necessary algebraic background is developped in 
      \cite[Ch.\ 12]{baues} and \cite[Introduction (B5-B6)]{baues}.

      Let $p$ for the moment be an arbitrary prime and $\GG=\ZZ/p^2$.
      %As in \cite[Ch.\ 12]{baues}
      We consider exact sequences of $\GG$-modules of the form
      \begin{align*}
        M_\bullet &= \Big(\xymatrix@1@C=1.5em{
          A^{\otimes m}\ar@{ >->}[r]^-{\iota}&M_1\ar[r]^-{\partial}
          & M_0\ar@{ ->>}[r]^-{\pi}&A^{\otimes m}
        }\Big)
      \end{align*}
      Under certain assumptions (e.g., if both factors are $[p]$-algebras
      in the sense of \cite[12.1.2]{baues}) one can define the folding product
      \begin{align*}
        (M\fotimes N)_\bullet &= \Big(
          A^{\otimes (m+n)}\to/ >->/^{\iota_\sharp}
          (M\fotimes N)_1\to^{\partial_\sharp}
          \underbrace{(M\fotimes N)_0}_{=M_0\otimes N_0}
          \to/ ->>/^{\pi\otimes\pi} A^{\otimes (m+n)}
        \Big)
      \end{align*}
      of two such sequences.
      Here $(M\fotimes N)_1$ is a quotient of $M_1\otimes N_0 \oplus N_0\otimes M_1$,
      so we can represent its elements as tensors $m\fotimes n$ 
      where either $m\in M_1$, $n\in N_0$ or $m\in M_0$, $n\in N_1$.
      Let $R_M=\ker \left(M_0\rightarrow A\right)$ 
      and $R_N=\ker \left(N_0\rightarrow A\right)$ be the relation modules.
      Then $(M\fotimes N)_1$ fits into the short exact sequence
      \begin{align*} 
        A^{\otimes (m+n)}\to/ >->/^{\iota_\sharp}(M\fotimes N)_1
        \to/ ->>/^{\partial} R_M\otimes N_0 + M_0\otimes R_N = R_{M\fotimes N}
      \end{align*}
      with $\partial(m \fotimes n) 
      = (\partial m)\otimes n + (-1)^{\grade{m}} m \otimes (\partial n)$.
      
      Unfortunately, $D_\bullet$ and $E_\bullet$ are not $[p]$-algebras
      in the sense of \cite[12.1.2]{baues}, 
      because  $D_0$ and $E_0$ fail to be $\GG$-free.
      It is easy to see, however,  that in both cases
      $\partial$ restricts to an isomorphism $\mu_0M_0\to pM_0$,
      so the reduction $\tilde M_\bullet$ 
      with $\tilde M_1=M_1/\mu_0M_0$ and $\tilde M_0 = M_0/pM_0$
      is again an exact sequence.
      A careful reading of Baues's theory shows that this suffices
      for the construction of the folding product.
      
      Assume now that we have a right-linear splitting $u:R_M\hookrightarrow M_1$
      of $\partial$.
      %$\partial: M_1\to/ ->>/R_M$.
      For $B_\bullet$ such a splitting has been 
      established in \cite[16.1.3-16.1.5]{baues}.
      For $D_\bullet$ we take the map $R_D\rightarrow D_1$
      \begin{align*}
        2\Sq(R) &\mapsto \mu_0\Sq(R),
        \quad Y_{k,l}a \mapsto U_{k,l}a\quad\text{(for $k<l$, $a\in A$)}.
      \end{align*} 
      from \eqref{udef} in the introduction.
      We extend this to $R_E=R_D\oplus W\rightarrow E_1 = D_1\oplus W$ 
      via $u_E=u_D\oplus\id_W$ where $W=X+\mu_0X$.
      We then get an induced splitting $u_\sharp$ for $(M\fotimes M)_\bullet$
      with $u_\sharp(r\otimes m) = u(r)\fotimes m$
      and  $u_\sharp(m\otimes r) = m\fotimes u(r)$
      for $r\in R_M$, $m\in M_0$.
      
      The splitting $u$ allows us to decompose $M_1$ as the direct sum 
      $M_1 = \iota(A) \oplus u(R_M)$. 
      However, this decomposition is only valid for the {\em right}
      action of $M_0$ on $M_\bullet$. We also have an action from the left
      and this is described by the associated {\em multiplication map}\footnote{
        This map is denoted $A$ in Baues's theory.}
      $\op:M_0\otimes R_M\rightarrow A^{\otimes m}$ with
      $$m\cdot u(r) = u(m\cdot r) + \iota(\op(m,r)).$$
      In our examples, $\op$ actually factors through 
      $M_0\otimes R_M\twoheadrightarrow A\otimes R_M$.
      For $B_\bullet$ this is proved in \cite[16.3.3]{baues}.
      For $D_\bullet$ and $E_\bullet$ it is obvious as both 
      $D_1$ and $E_1$ are $A$-bimodules to begin with. 
      
      We will now compute $\op$ and $\op_\sharp$ explicitly
      for $D_\bullet$ and $E_\bullet$.
     \begin{lemma}\label{opcomp}
        For $d\in D_0$ and $-1\le k<l$
        one has $\op(a,2d)=\kappa(a)\pi(d)$
        and
        \begin{align*}
          \op(a,Y_{k,l}) 
          &= \sum_{\substack{i,j\ge 0,\\k+i\ge l+j}} 
          \Sq(\Delta_{k+i+1}+\Delta_{l+j+1})\cont(\xi_i^{2^{k+1}}\xi_j^{2^{l+1}},a).
        \end{align*}
        Furthermore, $\op(a,x)=0$ for all $x\in X+\mu_0X$.
      \end{lemma}
       \begin{proof}
        Since $u(2d)=\mu_0\pi(d)$ one finds
        $au(2d)=\kappa(a)\pi(d)+u(a\cdot 2d)$ 
        which proves $\op(a,2d)=\kappa(a)\pi(d)$.
        
        %Now fix some $-1\le k<l$. 
        %For brevity, let $a_{i,j}=\cont(\xi_i^{2^{k+1}}\xi_j^{2^{l+1}},a)$.
        We have 
        $a\cdot u(Y_{k,l}) = 
        \sum_{i,j\ge 0} U_{k+i,l+j}$ $\cont(\xi_i^{2^{k+1}}\xi_j^{2^{l+1}},a)$.
        Using the relations \myref{urels} we can write 
        \begin{align*}
          U_{k+i,l+j} &= \begin{cases} 
            u(Y_{k+i,l+j}) & (k+i<l+j), \\
            u(2\Sq(\Delta_{k+i+2})) + \Sq(2\Delta_{k+i+1}) & (k+i=l+j), \\
            u(Y_{l+j,k+i}) + \Sq(\Delta_{k+i+1}+\Delta_{l+j+1})   & (k+i>l+j).
          \end{cases}
        \end{align*}
        Therefore 
        \begin{align*}
          a\cdot u(Y_{k,l})
          &= u(aY_{k,l}) + \sum_{\substack{i,j\ge 0,\\k+i\ge l+j}} 
          \Sq(\Delta_{k+i+1}+\Delta_{l+j+1})\cont(\xi_i^{2^{k+1}}\xi_j^{2^{l+1}},a)
        \end{align*}
        as claimed.

        Finally, $\op(a,-)$ vanishes on $M=X+\mu_0X$
        because $u\vert_M=\id$ is left-linear.
      \end{proof}

       For $\op_\sharp$ there is a similar result.
       \begin{lemma}\label{opsharpcomp}
        Write $B_{k,l,i,j}=\Sq(\Delta_{k+i+1}+\Delta_{l+j+1})$. Then
        \begin{align*}
          \op_\sharp(a,\Delta(2d)) &= \Delta\op(a,2d), 
          \quad \text{(for $d\in D_0$)}, \\
          \op_\sharp(a,\Delta(Y_{k,l})) 
          &= \sum_{\substack{i,j\ge 0,\\k+i\ge l+j}} 
          \left(B_{k,l,i,j}\otimes 1 + 1\otimes B_{k,l,i,j}\right) 
          \cont(\xi_i^{2^{k+1}}\xi_j^{2^{l+1}},a).
        \end{align*}
        One has $\op_\sharp(a,\Delta(x)) = 0$ for $x\in X+\mu_0X$.
      \end{lemma}
       \begin{proof}
        The first claim follows from
        $$\op_\sharp(a,\Delta(2d)) = \kappa(a)\Delta(2d) 
        = \Delta\left(\kappa(a)\cdot 2d\right) = \Delta\op(a,2d).$$
        For the second we use
        $\op_\sharp(a,\Delta(Y_{k,l})) 
        = \op_\sharp(a,Y_{k,l}\otimes 1 + 1\otimes Y_{k,l})$. 
        From Lemma \ref{opcomp} we find
        \begin{align*}
          \op_\sharp(a,Y_{k,l}\otimes 1) 
          &= \sum \op(a',Y_{k,l})\otimes a'' \\
          &= \sum B_{k,l,i,j}\cont(\,\cdots,a')\otimes a'' \\
          &= \sum (B_{k,l,i,j}\otimes 1) \cont(\,\cdots, a)
        \end{align*}
        where we have temporarily suppressed some details.
        There is a similar formula for $\op_\sharp(a,1\otimes Y_{k,l})$
        and together they make up the second claim.

        That $\op_\sharp(-,\Delta(X+\mu_0X))$ vanishes 
        is clear from the vanishing of $\op$ on $A\otimes (X+\mu_0X)$.
      \end{proof}
    \end{subsection}

    \begin{subsection}{The secondary coproduct}
      We can now define the secondary diagonal 
      $\Delta_\bullet:E_\bullet \rightarrow (E\fotimes E)_\bullet$.
      We still need a few preparations.
      \begin{lemma}\label{nabxlin}
        Let $U''\subset U$ be the sub-bimodule on the $U_{k,l}$ with $k,l\ge 0$.
        There is a bilinear $\nabla:U''\rightarrow A\otimes A$
        with $U_{k,l}\mapsto Q_l\otimes Q_k$.
      \end{lemma}
      \begin{proof}
        One has
        \begin{align*}
          a\left(Q_k\otimes 1\right) &=
          \sum (a'Q_k\otimes a'') = 
          \sum_{i\ge 0} Q_{k+i}\cont(\xi_i^{2^{k+1}},a') \otimes a'' \\
          &= \sum_{i\ge 0} (Q_{k+i}\otimes 1)\cont(\xi_i^{2^{k+1}},a).
          \intertext{Therefore}
          a\left(Q_k\otimes Q_l\right) 
          &= a\left(Q_k\otimes 1\right)\left(1\otimes Q_l\right) = \sum_{i,j\ge 0} 
          (Q_{k+i}\otimes Q_{l+j})\cont(\xi_i^{2^{k+1}}\xi_j^{2^{l+1}},a)
        \end{align*}
        which is the same commutation relation as for the $U_{k,l}$.
      \end{proof}

      \begin{lemma}\label{Phidef}
        There is a right-linear
        $\nabla:R_E\rightarrow A\otimes A \,\oplus\, \mu_0A\otimes A$
        with
        \begin{align*}
          \nabla X_{k,l} &= Q_l\otimes Q_k, 
          \quad \nabla \mu_0X_{k,l} = \mu_0 Q_l\otimes Q_k
          & (0&\le k,l) \\
          \quad \nabla Y_{k,l} &= Q_l\otimes Q_k
          & (0&\le k<l)
        \end{align*}
        and $\nabla\vert_{2D_0}=\nabla\vert_{Z_{\ast}}=0$
        where $Z_k=X_{-1,k}+Y_{-1,k}$.
        Let $\Phi(a,r) = \nabla(ar)-a(\nabla r)$ be 
        the left linearity defect of $\nabla$.
        Then
        \begin{align}
          \label{nablalindef}
          \Phi(a,r) &= \Delta \op(a,r) + \op_\sharp(a,\Delta r)
        \end{align}
        for $a\in A$ and $r\in R_E$. 
        %In particular, $\nabla\vert_{X+\mu_0X}$ is bilinear.
      \end{lemma}
      \begin{proof}
        $R_E$ is free as a right $A$-module 
        with basis $2$, $Z_k$ 
        (for $0\le k$), 
        $Y_{k,l}$ (for $0\le k<l$) and 
        $X_{k,l}$, $\mu_0X_{k,l}$ (for $0\le k,l$).
        Therefore $\nabla$ is well-defined and right-linear.

        We have 
        $\Phi(a,X_{k,l})=0$ and $\Phi(a,\mu_0X_{k,l})=0$ by Lemma \ref{nabxlin},
        $\Phi(a,2)=0$ and $\Delta \op(a,2) + \op_\sharp(a,\Delta 2)=0$
        by Lemma \ref{opsharpcomp},
        so it just remains to prove the formula for $r=Y_{k,l}$ and $r=Z_k$.

        Combining Lemmas \ref{opcomp} and \ref{opsharpcomp} we find
        \begin{align*}
          \Delta &\op(a,Y_{k,l}) + \op_\sharp(a,\Delta Y_{k,l}) \\
          &= \sum_{\substack{i,j\ge 0,\\k+i\ge l+j}} 
          \underbrace{\left(\Delta B_{k,l,i,j} 
            - B_{k,l,i,j}\otimes 1 + 1\otimes B_{k,l,i,j}\right) }_{=:C_{k,l,i,j}}
          \cont(\xi_i^{2^{k+1}}\xi_j^{2^{l+1}},a) \\
          \intertext{where}
          C_{k,l,i,j}
          &= \begin{cases}
            Q_{k+i+1}\otimes Q_{l+j+1} + Q_{l+j+1}\otimes Q_{k+i+1} & (k+i+1\not=l+j+1),\\
            Q_{k+i+1}\otimes Q_{l+j+1} & (k+i+1=l+j+1).
          \end{cases}
        \end{align*}
        To see that this is $\Phi(a,Y_{k,l})$ note first that 
        $\nabla(aU_{k,l})-a\nabla(U_{k,l})=0$ by Lemma \ref{nabxlin}.
        We can compute $\Phi(a,Y_{k,l})=\nabla(aY_{k,l})-a\nabla(Y_{k,l})$
        from this by changing every $\nabla U_{n,m}$ to $\nabla Y_{n,m}$.
        Since $\nabla U_{k,l} = \nabla Y_{k,l}$ for $k<l$ and
        $$\nabla U_{k+i,l+j} = 
        \begin{cases}
          \nabla Y_{k+i,l+j}  + C_{k,l,i,j} & (k+i\ge l+j)\\
          \nabla Y_{k+i,l+j}  & (k+i<l+j)
        \end{cases}$$
        this introduces exactly the error terms from the $C_{k,l,i,j}$.

        The case of $Z_k$ is similar and left to the reader.
      \end{proof}
      
      Now define $\coz$, $L:R_E\rightarrow A\otimes A$
      by $\nabla(r) = \coz(r) + \mu_0L(r)$. 
      Recall that $E_1 = \iota(A) \oplus u(R_E)$ and 
      let $\Delta_1:E_1 \rightarrow (E\fotimes E)_1$ be 
      given by 
      \begin{align}\label{deltaonee}
        \Delta_1\left(\iota(a)\right) &= \iota_\sharp\left(\Delta(a)\right) 
        ,\quad
        \Delta_1\left(u(r)\right) = u_\sharp\left(\Delta_0(r)\right) 
        + \iota_\sharp\left(\coz(r)\right).
      \end{align}
      \begin{lemma}
        With this coproduct
        $E_\bullet$ becomes a secondary Hopf algebra.
      \end{lemma}
      \begin{proof}
        First note that $\Delta_1$ is right-linear and fits 
        into a commutative diagram
        \begin{align*}
          \xymatrix{
            A \ar@{ >->}[r]^-{\iota} \ar[d]^-{\vphantom{\Delta_1}\Delta} 
            & E_1 \ar[r]^-{\partial}\ar[d]^-{\Delta_1}  
            & E_0 \ar@{ ->>}[r]\ar[d]^-{\Delta_0}   
            & A\ar[d]^-{\vphantom{\Delta_1}\Delta}  \\
            A\otimes A \ar@{ >->}[r]^-{\iota_\sharp} 
            & (E\fotimes E)_1 \ar[r]^-{\partial} 
            & E_0\otimes E_0 \ar@{ ->>}[r] & A\otimes A \\
          }
        \end{align*}
        $\Delta_\bullet:E_\bullet\rightarrow (E\fotimes E)_\bullet$ 
        is therefore a map of $[p]$-algebras
        in the sense of \cite[12.1.2 (4)]{baues}.
        There is also a natural augmentation $\aug_\bullet : E_\bullet
        \rightarrow G_\bullet$ where 
        $G_\bullet = 
        \left(\FF\hookrightarrow\FF+\mu_0\FF\rightarrow \GG\twoheadrightarrow\FF\right)$
        is the unit object for the folding product. 
        
        It remains to verify the usual identities
        \begin{align*}
          (\aug_\bullet \fotimes\id)\Delta_\bullet &= \id
          = (\id\fotimes\aug_\bullet)\Delta_\bullet,
          \quad
          (\Delta_\bullet\fotimes\id)\Delta_\bullet
          = (\id\fotimes \Delta_\bullet)\Delta_\bullet.
        \end{align*}
        This can be done on the $A$ generators 
        $\mu_0, U_{k,l}, X_{k,l}, \mu_0X_{k,l} \in E_1$.
        We have $\Delta_1(\mu_0) = \mu_0\fotimes 1 = 1\fotimes \mu_0$ and
        \begin{align*}
          \Delta_1\left(U_{k,l}\right) 
          &= U_{k,l}\fotimes 1 + 1\fotimes U_{k,l} + Q_l\fotimes Q_k, 
          &&\text{(for $k<l$)}\\
          \Delta_1\left(X_{k,l}\right) 
          &= X_{k,l}\fotimes 1 + 1\fotimes X_{k,l} + Q_l\fotimes Q_k, \\
          \Delta_1\left(\mu_0X_{k,l}\right) 
          &= \mu_0X_{k,l}\fotimes 1 + 1\fotimes \mu_0X_{k,l}.
        \end{align*}
        Then, for example, 
        \begin{align*}
          (\id\fotimes\Delta_1)\Delta_1\left(U_{k,l}\right)
          &= (\id\otimes\Delta_1)\left(U_{k,l}\fotimes 1 + 1\fotimes U_{k,l} + Q_l\fotimes Q_k\right) \\
          &= U_{k,l}\fotimes 1\fotimes 1
          + 1 \fotimes U_{k,l}\fotimes 1
          + 1 \fotimes 1 \fotimes U_{k,l} \\
          &\qquad+ 1\fotimes Q_l\fotimes Q_k
          + Q_l\fotimes 1\fotimes Q_k
          + Q_l\fotimes Q_k\fotimes 1 \\
          &= (\Delta_1\fotimes\id)\Delta_1\left(U_{k,l}\right).
        \end{align*}
        We leave the remaining cases to the reader.
      \end{proof}

      Our $\Delta_1$ fails to be left-linear or symmetric;
      as in \cite[14.1]{baues} that failure is captured by 
      the {\em left action operator} $L$
      and the {\em symmetry operator} $S$ 
      as defined in the following Lemma.
      \begin{lemma}\label{elscomp}
        For $e\in E_1$ and $a\in A$ one has
        $$\Delta_1(ae) = a\Delta_1(e) 
        + \iota_\sharp\left(\kappa(a)L(\partial e)\right),
        \quad
        T\Delta_1(e) = \Delta_1(e) + \iota_\sharp\left(S(\partial e)\right)$$
        with $S(r) = (1+T)\coz(r)$
        where $T:A\otimes A\rightarrow A\otimes A$ is the twist map.
      \end{lemma}
      \begin{proof}
        That $S(r) = (1+T)\coz(r)$ is obvious from the definition.
        For the left-linearity defect one computes
        \begin{align*}
          \Delta_1\left(a\cdot u(r)\right) &=
          \Delta_1\left(u(ar) + \iota\left(\op(a,r)\right)\right) \\
          &= u_\sharp\left(\Delta_0(ar)\right) 
          + \iota_\sharp\left(\coz(ar) + \Delta\op(a,r)\right), \\
          a\cdot \Delta_1\left(u(r)\right) &=
          a\cdot\left(u_\sharp\left(\Delta_0(r)\right) 
          + \iota_\sharp\left(\coz(r)\right)\right)
          \\
          &= u_\sharp\left(a\cdot \Delta_0(r)\right) 
          + \iota_\sharp\left(\op_\sharp\left(a,\Delta_0(r)\right) 
          + a\cdot \coz(r)\right). 
        \end{align*}
        Therefore $\Delta_1\left(au(r)\right) - a\Delta_1\left(u(r)\right)$
        is
        \begin{gather*}
          \iota_\sharp\left(\coz(ar)-a\coz(r)
          +\Delta\op(a,r)-\op_\sharp\left(a,\Delta_0(r)\right)\right)
          \intertext{which by Lemma \ref{Phidef} is}
          \iota_\sharp\left(\coz(ar)-a\coz(r)+\nabla(ar)-a\nabla(r)\right) 
          = \iota_\sharp\left(\kappa(a)L(r)\right).
        \end{gather*}
      \end{proof}
      Note that in Baues's book $L$ was originally defined as a certain map
      $L:A\otimes R\rightarrow A\otimes A$.
      However, it was shown in \cite[12.7]{jiblmpi} 
      that $L(a\otimes r) = \kappa(a)L(\Sq^1\otimes r)$,
      so our $L(r)$ corresponds to $L(\Sq^1\otimes a)$ in \cite{baues}.   
    \end{subsection}
    
    \begin{subsection}{Proof of 
        \texorpdfstring{$B_\bullet\sim E_\bullet$}{the equivalence}}
      We are now very close to establishing the weak equivalence
      between $E_\bullet$ and the secondary Steenrod algebra $B_\bullet$.
      Recall that $B_0$ is the free 
      associative algebra over $\ZZ/4$ on
      the $\Sq^k$ with $k>0$. 
      Let $\cmap_0:B_0\rightarrow E_0$ be the multiplicative 
      map with $B_0\owns\Sq^n\mapsto\Sq^n\in D_0$.
      It's easily checked that $\cmap_0$ is also comultiplicative.
      
      Let $\cmap_0\os E_1$ be defined as the pullback of $E_1\rightarrow E_0$
      along $\cmap_0$. We then have a commutative diagram
      \begin{align*}
        \xymatrix{
          A \ar@{ >->}[r] & E_1 \ar[r] & E_0 \ar@{ ->>}[r] & A \\
          A\ar@{=}[u] \ar@{ >->}[r] & \cmap_0\os E_1 \ar[r]\ar[u]^-{\cmap_1}
          & B_0 \ar@{ ->>}[r]\ar[u]^-{\cmap_0} & A\ar@{=}[u] \\           
        }
      \end{align*}
      that defines a new sequence $\cmap\os E_\bullet$
      together with a weak equivalence 
      to $E_\bullet$.
      %$(\cmap\os E)_\bullet \xrightarrow{\sim} E_\bullet$.
      We will prove that $\cmap\os E_\bullet \cong B_\bullet$.

      \begin{lemma}
        $\cmap\os E$ inherits a secondary Hopf algebra structure from $E_\bullet$
        such that the map $\cmap\os E_\bullet \rightarrow E_\bullet$
        is a map of secondary Hopf algebras.
      \end{lemma}
      \begin{proof}
        Indeed, using the splitting $(\cmap\os E\fotimes \cmap\os E)_1 
        = \iota'_\sharp\left(A\otimes A\right) \oplus 
        u'_\sharp\left(R_{B\otimes B}\right)$
        we can transport the definition \myref{deltaonee} to
        \begin{align*}
          \Delta_1\left(\iota'(a)\right) &= \iota'_\sharp\left(\Delta(a)\right) 
          ,\quad
          \Delta_1\left(u'(r)\right) = u'_\sharp\left(\Delta_0(r)\right) 
          + \iota'_\sharp\left(\coz(\cmap_0(r))\right).
        \end{align*}
        We leave the details to the reader.
      \end{proof}

      Note that the left action and symmetry operators
      of $\cmap\os E_\bullet$ 
      are given by $L'=L\circ\cmap_0$ and $S'=S\circ\cmap_0$.
      The following Lemma therefore shows that these agree
      with the operators from the secondary 
      Steenrod algebra.
      \begin{lemma}\label{lsbcomp}
        Decompose 
        $\nabla\cmap_0\vert_{R_B}:R_B\rightarrow A\otimes A\,\oplus\,\mu_0 A\otimes A$
        as 
        \begin{align*}
          \nabla\left(\cmap_0(r)\right) 
          &= \coz(r) + \mu_0L(r)\quad\text{with $\coz$, $L:R_B\rightarrow A\otimes A$.}
        \end{align*}
        Then
        $r\mapsto L(r)$ resp.\ $r\mapsto (1+T)\coz(r)$ 
        coincide with the left-action resp.\ symmetry operator of $B_\bullet$.
      \end{lemma}
      \begin{proof}
        For $0<n<2m$ let $[n,m]\in R_B$ denote the Adem relation 
        \begin{align*}
          \underbrace{\Sq^n\otimes \Sq^m + 
            \sum_{1\le k \le \frac{n}2} \binom{m-k-1}{n-2k} \Sq^{m+n-k}\otimes \Sq^k}_{= \langle n, m\rangle}
          + \underbrace{\vphantom{\sum_{1\le k \le \frac{n}2}}\binom{m-1}{n} \Sq^{m+n}}_{= \Lambda_{n,m}}.
        \end{align*}
        Together with $2\in R_B$ the $[n,m]$ generate $R_B$ as a $B_0$-bimodule.
        We let $F^1=\ZZ/2\{\Sq^n\vert n\ge 1\}$, so 
        $\langle n, m\rangle\in F^1\otimes F^1$ and $\Lambda_{n,m}\in F^1$.

        According to \cite[12.7]{jiblmpi} or \cite[14.4.3]{baues}
        the left action map is the unique 
        bilinear $L:R_B\rightarrow A\otimes A$
        with 
        $L([n,m]) = L_R\left(\langle n, m\rangle\right)$ 
        where $L_R:F^1\otimes F^1\rightarrow A\otimes A$
        is given by
        \begin{align*}
          L_R\left(\Sq^n\otimes\Sq^m\right) &=  \sum_{\substack{n_1+n_2=n\\m_1+m_2=m\\m_1,n_2\,\text{odd}}}
          \Sq^{n_1}\Sq^{m_1}\otimes\Sq^{n_2}\Sq^{m_2}.
        \end{align*}
        Lemma \ref{Phidef} proves that the $L$ that we extracted from $\nabla$
        is also bilinear, so we only have to verify that it gives the right value
        on the Adem relations.
        We now compute
        \begin{align}\label{sqsq1}
          \nonumber
          \Sq^n\ast\Sq^m 
          &= \Sq^n\Sq^m + \psi(\Sq^n)\psi(Sq^m)\mu_0 + X_{-1}\psi(\Sq^n)\kappa(\Sq^{m}) \\
          &= \Sq^n\Sq^m + X_0\Sq^{n-1}X_0\Sq^{m-1}\mu_0 + X_{-1,0}\Sq^{n-1}\Sq^{m-1}.
        \end{align}
        For the $\mu_0$-component we then find
        \begin{align*}
          \nabla(X_0\Sq^{n-1}&X_0\Sq^{m-1})
          = \left((1\otimes Q_0)\Sq^{n-1}\right)\cdot\left((Q_0\otimes 1)\Sq^{m-1}\right)\\
          &= \Big(\sum_{\substack{n_1+n_2=n,\\\text{$n_2$ odd}}}\Sq^{n_1}\otimes\Sq^{n_2}\Big)
          \cdot
          \Big(\sum_{\substack{m_1+m_2=m,\\\text{$m_1$ odd}}}\Sq^{m_1}\otimes\Sq^{m_2}\Big)
        \end{align*}
        as claimed.

        The identification of $S = (1+T)\coz$ with the symmetry operator 
        proceeds similarly. We first evaluate $S([n,m])$. 
        Moving $\mu_0$ to the right gives 
        $$\nabla(\cmap_0(r)) = \mu_0L(r) + \coz(r) = L(r)\mu_0 + 
        \underbrace{\kappa(L(r))+\coz(r)}_{=:\tilde\coz(r)}.$$
        We claim that $\Sq^n\Sq^m\in D_0$ does not have any $Y_{k,l}$-component
        with $0\le k,l$. Indeed, from the coproduct formula in $D_0$ we find
        $$\text{$\Delta \xi_{n,m} \equiv \xi_n\xi_m\otimes\xi_1$
          mod $\xi_{k,l}\otimes 1$, $1\otimes\xi_{k,l}$, 
          $1\otimes\xi_j$ with $j\ge 2$}.$$
        From \myref{sqsq1} we then find
        $$\tilde\coz(\Sq^n\Sq^m) 
        = \nabla\Sq^n\Sq^m + \nabla X_{-1,0}\Sq^{n-1}\Sq^{m-1} = 0.$$
        It follows that 
        $S([n,m])=(1+T)\kappa\left(L([n,m])\right)
        =(1+T)L\left(\kappa([n,m])\right)$.
        We still need to show that this is the expected outcome.
        Let $\langle n,m\rangle = \sum_i\Sq^{n_i}\otimes\Sq^{m_i}$.
        Expanding slightly on the computation above, we see that
        \begin{align*}
          L([n,m]) 
          &= \nabla\left(X_{0,0}\Sq^{n_i-1}\Sq^{m_i-1}+X_{0,1}\Sq^{n_i-3}\Sq^{m_i-1}\right).
          \intertext{Therefore}
          (1+T)L(\kappa([n,m])) 
          &= (1+T)\nabla X_{0,1}\left(\Sq^{n_i-4}\Sq^{m_i-1}+\Sq^{n_i-3}\Sq^{m_i-2}\right)
        \end{align*}
        where we have ignored the $X_{0,0}(\cdots)$ because $(1+T)\nabla X_{0,0}=0$.
        Since $\Lambda_{n,m}$ $=$ $\sum_i\Sq^{n_i}\Sq^{m_i}$ $\in$ $F^1$ we have
        \begin{align*}
          0 &= \cont(\xi_2,\sum_i\Sq^{n_i}\Sq^{m_i}) = \sum_i\Sq^{n_i-2}\Sq^{m_i-1}, \\
          0 &= \cont(\xi_1^2,\cont(\xi_2,\sum_i\Sq^{n_i}\Sq^{m_i})) = 
          \sum_i \left(\Sq^{n_i-4}\Sq^{m_i-1} + \Sq^{n_i-2}\Sq^{m_i-3}\right).
        \end{align*}
        We finally arrive at
        \begin{align*}
          (1+T)L(\kappa([n,m]))
          &= (1+T)\nabla X_{0,1}\left(\Sq^{n_i-2}\Sq^{m_i-3}+\Sq^{n_i-3}\Sq^{m_i-2}\right).
        \end{align*}
        In the notation of the remark following \cite[16.2.3]{baues} 
        this is just $(1+T)K[n,m]$
        where it is also affirmed that this is 
        the correct value for $S\left([n,m]\right)$. 

        The proof of the Lemma will be complete, once we have verified that
        $S$ has the right linearity properties. From Lemma \ref{Phidef}
        we see that the linearity defect of $\nabla$ is symmetrical;
        therefore $(1+T)\nabla = S + \mu_0(1+T)L$
        is actually bilinear.
        For $S$ this translates into 
        \begin{align*}
          S(ra) &= S(r)a, \quad S(ar) = aS(r) + (1+T)\kappa(a)L(r).
        \end{align*}
        This agrees with the characterization in \cite[14.5.2]{baues}.
      \end{proof}
      
      \begin{coro}
        There is an isomorphism $\cmap\os E_\bullet\cong B_\bullet$.
      \end{coro}
      \begin{proof}
        Apply the Uniqueness Theorem \cite[15.3.13]{baues}. 
      \end{proof}
      This also proves Theorem
      \ref{mainthm} since we have by construction 
      a chain of weak equivalences
      $\cmap\os E_\bullet\xrightarrow{\sim} E_\bullet \xrightarrow{\sim} D_\bullet$.

      \begin{remark}
        The map $S:R_E\rightarrow A\otimes A$ does not factor
        through the projection $R_E\rightarrow R_D$.
        This can be seen from
        the computation
        \begin{align*}
          [3,2] &= 2 \Sq(2,1) + 2 \Sq(5) + (X_{-1,0}+Y_{-1,0}) (\Sq(0,1) + \Sq(3)) \\
          &+ X_{0,0} \Sq(2) + X_{0,1}
          + \mu_0 X_{0,0} (\Sq(0,1) + \Sq(3)) + \mu_0 X_{0,1} \Sq(1),
          \\
          [2,2]\Sq^1 &= 2 \Sq(2,1) + 2 \Sq(5) 
          + (X_{-1,0}+Y_{-1,0}) (\Sq(0,1) + \Sq(3)) \\
          &+ \mu_0 X_{0,0} (\Sq(0,1) + \Sq(3)) + \mu_0 X_{0,1} \Sq(1).
        \end{align*}
        One finds that $S([3,2]) = Q_1\otimes Q_0+ Q_0\otimes Q_1$
        and $S([2,2]\Sq^1) = 0$
        even though $[3,2]$ and $[2,2]\Sq^1$ have the same image in $D_0$.
        This shows that the secondary diagonal 
        $\Delta_1:B_1\rightarrow (B\fotimes B)_1$
        has no analogue over $D_\bullet$.
      \end{remark}

    \end{subsection}
  \end{section}

%%%%%%%%%%%%%%%%%%%%%%%%%%%%%%%%%%%%%%%%%%%%%%%%%%%%%%%%%%%%%%%%%%%%%%%%%%%%%%
%%%%%%%%%%%%%%%%%%%%%%%%%%%%%%%%%%%%%%%%%%%%%%%%%%%%%%%%%%%%%%%%%%%%%%%%%%%%%%
%%%%%%%%%%%%%%%%%%%%%%%%%%%%%%%%%%%%%%%%%%%%%%%%%%%%%%%%%%%%%%%%%%%%%%%%%%%%%%
%%%%%%%%%%%%%%%%%%%%%%%%%%%%%%%%%%%%%%%%%%%%%%%%%%%%%%%%%%%%%%%%%%%%%%%%%%%%%%

\newcommand{\ZZp}{{\ZZ_{(p)}}}
\newcommand{\FFp}{{\FF_{p}}}
\newcommand{\degree}[1]{{\vert{#1}\vert}}
  \begin{appendix}
    \begin{section}{\texorpdfstring{$\EBP$}{EBP} and a model at odd primes}
      Let $p$ be a prime
      and let $\BP$ denote the Brown-Peterson spectrum at $p$.
      In this appendix we show how a model 
      of the secondary Steenrod algebra can be extracted from $\BP$ if $p>2$.

      Recall that the homology $H\us\BP$ is the polynomial 
      algebra over $\ZZp$ on generators $(m_k)_{k=1,2,\ldots}$
      and that $\BP\us \subset H\us\BP$ is the subalgebra
      generated by the Araki generators $(v_k)_{k=1,2,\ldots}$.
      Let $\EBP\us = E(\mu_k\,\vert\,k\ge 0)\otimes\BP\us$ 
      with exterior algebra generators $\mu_k$ of degree
      $\degree{\mu_k}=\degree{v_k}+1$.
      $\EBP\us$ is a free $\BP\us$-module and
      defines a Landweber exact homology theory $\EBP$. 
      Obviously, the
      representing spectrum is just a wedge of copies of $\BP$.
      %, with one copy for every monomial $\mu^\epsilon\in E$.
      As usual, we let $I=(v_k)\subset \BP\us$ be the maximal invariant ideal.
      
      The cooperation Hopf algebroid $\EBP\us\EBP$ 
      is very easy to compute:
      \begin{lemma}\label{ebpprops}
        One has 
        $\EBP\us\EBP = E(\mu_k)\otimes_\ZZp \BP\us\BP \otimes_\ZZp E(\tau_k)$
        with 
        \begin{align}\label{taudef}
          \eta_R(\mu_n) &= \sum_{k=0}^n \mu_k t_{n-k}^{p^k} + \tau_n
        \end{align}
        and
        \begin{align*}
          \Delta \tau_n &= 1\otimes \tau_n + \sum_{k=0}^n \tau_k\otimes t_{n-k}^{p^k} 
          + \sum_{0\le a\le n} \mu_a\left(
          - \Delta t_{n-a}^{p^a}
          + \!\!\sum_{b+c=n-a}\!\!t_b^{p^a}\otimes t_c^{p^{a+b}} \right).
        \end{align*}
        The other structure maps are inherited from $\BP\us\BP$.
      \end{lemma}
      \begin{proof}
        We use \myref{taudef} to define the 
        $\tau_k\in \EBP\us\EBP = E(\mu_k)\otimes \BP\us\BP\otimes E(\mu_k)$.
        $\Delta\tau_n$ 
        can then be computed from
        $(\eta_R\otimes\id)\eta_R(\mu_n) = \Delta\eta_R(\mu_n)$.
      \end{proof}

      We can put a differential on $\EBP$ by setting $\partial \mu_k = v_k$
      and this turns  $\EBP\us\EBP$ into a 
      differential Hopf algebroid.
      \begin{coro}
        For $p>2$ the homology Hopf algebroid of $\EBP\us\EBP$ 
        with respect to $\partial$ 
        is the dual Steenrod algebra $A\us$.
      \end{coro} 
      \begin{proof}
        We have $\partial \tau_n = \eta_R(v_n)-\sum_{k=0}^n v_k t_{n-k}^{p^k} \equiv 0$ 
        mod $I^2$, so there are $\tau_n'\equiv \tau_n$ mod $I$ with $\partial \tau_n' = 0$.
        Therefore $H\os\left(\EBP\us;\partial\right) = \FFp$
        and 
        \begin{align*}
          H\os\left(\EBP\us\EBP;\partial\right) 
          &= \FFp[t_k\vert k\ge 1]\otimes E(\tau_n'\vert n\ge 0) = A\us.
        \end{align*}
        Lemma \ref{ebpprops} then shows that the induced coproduct on $A\us$
        coincides with the usual one.
      \end{proof}

      We prefer to work with operations rather
      than cooperations.
      Write $E=\EBP\us$, $\Gamma\us = \EBP\us\EBP$
      and let 
      $\Gamma 
      = \Hom_{E}\left(\Gamma\us,E\right)$
      be the operation algebra $\EBP\os\EBP$ of $\EBP$.
      Then $\Gamma$ is a differential algebra and 
      for odd $p$ its homology $H(\Gamma;\partial)$
      can be identified with the Steenrod algebra $A$.
      We therefore get an exact sequence $P_\bullet$
      \begin{align}\label{ebpmod}
        \xymatrix{
          A \ar@{ >->}[r] & \coker\partial \ar@{->}[r]^-{\partial}
          & \ker \partial \ar@{ ->>}[r] & A.
        }
      \end{align}
      by splicing
      $H(\Gamma;\partial)\hookrightarrow \sfrac{\Gamma}{\im\partial}
      \twoheadrightarrow \im\partial$
      and
      $\im\partial\hookrightarrow \ker\partial \twoheadrightarrow H(\Gamma;\partial)$.
      We claim that for odd $p$ this sequence is a 
      model for the secondary Steenrod algebra.
      \begin{thm}\label{appmainthm}
        Let $p>2$ and let $B_\bullet\rightarrow G_\bullet$ 
        be the secondary Steenrod algebra with its canonical augmentation
        to $G_\bullet = (\FF_p\hookrightarrow \FF_p\{1,\mu_0\}
        \rightarrow \ZZp\twoheadrightarrow \FF_p)$.
        Then there is a diagram of crossed algebras 
        \begin{align}\label{ebpzigzag}
          \begin{array}{c}
            \xymatrix{
              P_\bullet \ar@{ ->>}[r] \ar[d]& (P/J^2)_\bullet \ar[d]
              & T_\bullet\ar@{ >->}[l] \ar[d]& B_\bullet \ar[l]\ar[d] \\
              {G^P}_\bullet \ar@{ ->>}[r] & {G^{P/J^2}}_\bullet 
              & {G^T}_\bullet\ar@{ >->}[l] & {G^{\vphantom{P}}}_\bullet \ar[l] 
            }
          \end{array}
        \end{align}
        where all horizontal maps are weak
        equivalences.
        %For $p>2$, $P_\bullet$ is weakly equivalent to $B_\bullet$.
      \end{thm}
      
      Note that $P_\bullet$ itself cannot be the target
      of a comparison map from $B_\bullet$ 
      as $p^2$ is zero in $B_0$ but
      not in $P_0$.
      In the statement we have also 
      singled out an intermediate sequence
      $T_\bullet$. This sequence is of independent interest because it is 
      quite small and given by explicit formulas. 
      
      To construct \myref{ebpzigzag} we first establish
      the diagram of augmentations. 
      %This is done in the following Lemma. 
      Let $J=I\cdot E \subset E$.
      \begin{lemma}\label{augzigzaglem}
        Let $ZE = \ker E\xrightarrow{\partial} E$
        and $w_k = v_k\mu_0 -p\mu_k = - \partial(\mu_0\mu_k)\in J$.      
        Then there is a commutative diagram
        \begin{align*}
          \xymatrix{
            {G^P}_\bullet \ar@{ ->>}[d] 
            & \FF_p\ar@{=}[d] \ar@{ >->}[r] 
            & E/\partial E \ar[r]^-{\partial} \ar@{ ->>}[d]
            & ZE \ar@{->>}[r] \ar@{ ->>}[d] &\FF_p \ar@{=}[d]
            \\
            {G^{P/J^2}}_\bullet  
            & \FF_p \ar@{ >->}[r] 
            & E/J \ar[r]^-{\partial} 
            & *+<3mm>{\Big(\ker E/J^2\raisebox{0pt}[0pt][0pt]{$\xrightarrow{\,\,\partial\,\,}$} E/J^3\Big)} \ar@{->>}[r] &\FF_p 
            \\
            {G^T}_\bullet \ar@{ >->}[u]  
            & \FF_p\ar@{=}[u] \ar@{ >->}[r] 
            & \FF_p\{1,\mu_k,\mu_0\mu_k\} \ar[r]^-{\partial} \ar@{ >->}[u]
            & \ZZp[v_k,w_k]/J^2 \ar@{->>}[r]  \ar@{ >->}[u] &\FF_p \ar@{=}[u]
            \\
            G_\bullet \ar@{ >->}[u] 
            & \FF_p\ar@{=}[u] \ar@{ >->}[r] 
            & \FF_p\{1,\mu_0\} \ar[r]^-{\partial} \ar@{ >->}[u]
            & \ZZ/(p^2) \ar@{->>}[r]  \ar@{ >->}[u] &\FF_p \ar@{=}[u]
          }
        \end{align*}
        with exact rows.
      \end{lemma}
      \begin{proof}
        This is straightforward, except for the 
        exactness of ${G^{P/J^2}}_\bullet$. First note that 
        \begin{align*}
          \xymatrix{
            \FF_p \ar@{ >->}[r] 
            & J/J^2\ar[r]^-{\partial} 
            & J^2/J^3\ar[r]^-{\partial} 
            & J^3/J^4\ar[r]^-{\partial} 
            &\cdots
          }
        \end{align*}
        is exact because it can be identified with the 
        super deRham complex
        $\Omega^n = \FF_p\{\mu^\epsilon d\mu_{i_1}\cdots d\mu_{i_n}\}$
        with $df = \sum \frac{\partial f}{\partial \mu_k} d\mu_k$
        via $v_k=d\mu_k$.
        Let $E_J$ denote the complex 
        \begin{align*}
          \xymatrix{
            E/J\ar[r]^-{\partial} 
            & E/J^2\ar[r]^-{\partial} 
            & E/J^3\ar[r]^-{\partial} 
            & E/J^4\ar[r]^-{\partial} 
            &\cdots.
          }
        \end{align*}
        Its associated graded 
        with respect to the $J$-adic filtration is the sum of shifted copies
        $\Omega^{k+\ast}$ for $k\ge 0$,
        so one has $H_k(E_J)=\FF_p$ for all $k$.
        The exactness of 
        $\,\FF_p\hookrightarrow E/J\rightarrow\left(\ker \,\partial:E/J^2\rightarrow E/J^3\right)\twoheadrightarrow\FF_p\,$
        %${G^{P/J^2}}_\bullet$ 
        is an easy consequence.
      \end{proof}

      Now let $P(R)Q(\epsilon)\in \Gamma = \Hom_E(\Gamma\us,E)$ 
      denote the dual of $t^R\tau^\epsilon$
      with respect to the monomial basis of $\Gamma\us$. 
      (One easily verifies that this
      is indeed the product of $P(R):=P(R)Q(0)$ 
      and $Q(\epsilon):=P(0)Q(\epsilon)$ as suggested by the notation.)
      We can think of $\Gamma$
      as the set $E\{\!\{ P(R)Q(\epsilon) \}\!\}$ of infinite sums 
      $\sum a_{R,\epsilon}P(R)Q(\epsilon)$ with coefficients
      $a_{R,\epsilon}\in E$.  
    
      It is important to realize that the $P(R)$ are {\em not} 
      $\partial$-cycles:
      for $p=2$, for example, one finds that $\partial\tau_n\equiv v_{n-1}^2t_1$ 
      mod $I^3$
      which shows that
      $\partial P^1 \equiv 4Q(0,1) + v_1^2Q(0,0,1) + \cdots$
      mod $I^3$.
      \begin{lemma}\label{partialtau}
        Let $p>2$. Then $\partial\tau_n\equiv 0$ mod $I^3$.
      \end{lemma}
      \begin{proof} 
        The claim is equivalent to 
        $\eta(v_n)\equiv \sum_{0\le k\le n} v_kt_{n-k}^{p^k}$ mod $I^3$.
        We leave this as an exercise.
      \end{proof}

      The following Lemma defines $(P/J^2)_\bullet$ and its weak equivalence 
      with $P_\bullet$.
      \begin{lemma}
        Let $Z\Gamma = \ker \partial:\Gamma\rightarrow\Gamma$.
        There is a commutative diagram 
        \begin{align*}
          \xymatrix{
            {P}_\bullet \ar@{ ->>}[d] 
            & A\ar@{=}[d] \ar@{ >->}[r] 
            & \Gamma/\partial \Gamma \ar[r]^-{\partial} \ar@{ ->>}[d]
            & Z\Gamma \ar@{->>}[r] \ar@{ ->>}[d] &A \ar@{=}[d]
            \\
            {P/J^2}_\bullet  
            & A \ar@{ >->}[r] 
            & \Gamma/J\Gamma \ar[r]^-{\partial} 
            & *+<3mm>{\Big(\ker \Gamma/J^2\Gamma\raisebox{0pt}[0pt][0pt]{$\xrightarrow{\,\,\partial\,\,}$} \Gamma/J^3\Gamma\Big)} \ar@{->>}[r] &A 
            }
        \end{align*}
        with exact rows.
      \end{lemma}
      \begin{proof}
        Choose $\tilde\tau_k\in\Gamma\us$ with $\tilde\tau_k\equiv\tau_k$ mod $I$ 
        and $\partial\tilde\tau_k=0$. Let $X(R;\epsilon)\in\Gamma$ be dual
        to $t^R\tilde\tau^\epsilon$. Then 
        $\Gamma=\prod_{R,\epsilon} E\cdot X(R;\epsilon)$ and
        $\partial X(R;\epsilon)=0$. It follows that the exactness
        can be checked on the coefficients alone where it was
        established in Lemma \ref{augzigzaglem}.
      \end{proof}

      The construction of $T_\bullet$ requires a more explicit 
      understanding of $\Gamma\us/I^2$.
      \begin{lemma}\label{ebpmodisq}
        For a family $(x_k)$ let $\gppol_{p^n}(x_k)\in \FFp[x_k]$ 
        %\marginpar{\Large\bf sign?}
        be defined by
        $\sum x_k^{p^n} - \left(\sum x_k\right)^{p^n} = p\gppol_{p^n}(x_k).$
        Then modulo $I^2$ one has 
        \begin{align*}
          \Delta t_n &\equiv \sum_{n=a+b}t_a\otimes t_b^{p^a} 
          + \sum_{0<k\le n} v_k \gppol_{p^k}\left( 
          t_a\otimes t_b^{p^a} \,\Big\vert\, a+b=n-k\right).
        \end{align*}
        Let $w_k=-\partial(\mu_0\mu_k) = v_k\mu_0 - p\mu_k$.
        Then 
        \begin{align*}
          \Delta\tau_n &\equiv 1\otimes \tau_n + \sum_{n=a+b} \tau_a\otimes t_{b}^{p^a} 
          + \sum_{0<k\le n} w_k \gppol_{p^k}\left( 
          t_a\otimes t_b^{p^a} \,\Big\vert\, a+b=n-k\right).
        \end{align*} 
        Furthermore,
        \begin{align*}
          \eta_R(v_n) &\equiv \sum_{0\le k\le n}v_kt_{n-k}^{p^k}, \\
          \eta_R(w_n) &\equiv -p\tau_n 
          + \sum_{1\le k<n}w_kt_{n-k}^{p^k} + \sum_{0\le k\le n}v_kt_{n-k}^{p^k}\tau_0,
        \end{align*}
      \end{lemma}
      \begin{proof}
        The $v_k$ are defined by $pm_n=\sum_{n=a+b}m_av_b^{p^a}$
        and it follows easily that $v_n\equiv pm_n$ modulo $I^2\cdot H\us(\EBP)$.
        Recall that $\eta_R(m_n) = \sum_{n=a+b} m_at_b^{p^a}$
        and that $\Delta t_n$ can be computed from 
        $(\eta_R\otimes\id)\eta_R(m_n) = \Delta\eta_R(m_n)$.
        Inductively, this gives
        \begin{align*}
          \Delta t_n &= \sum_{n=a+b} t_a\otimes t_b^{p^a} 
          + \sum_{0<k\le n} m_k
          \left( 
          -\Delta t_{n-k}^{p^k} 
          + \!\!\sum_{n-k=a+b}\!\! t_a^{p^k}\otimes t_b^{p^{k+a}} 
          \right) \\
          &\equiv \sum_{n=a+b} t_a\otimes t_b^{p^a} 
          + \sum_{0<k\le n} v_k \gppol_{p^k}\left( 
          t_a\otimes t_b^{p^a} \,\Big\vert\, a+b=n-k\right)
        \end{align*}
        as claimed.
        The formula for $\Delta\tau_n$ now 
        follows with Lemma \ref{ebpprops}.
        We leave the computation of $\eta_R(v_n)$ and $\eta_R(w_n)$
        to the reader.
      \end{proof}
     
      Let $S_\bullet = {G^T}_\bullet$ and recall that
      \begin{align*}
        S_0 &= \ZZ/p^2 + \FF_p\{v_k, w_k\,\vert\,k\ge 1\} \subset E/J^2, \\
        S_1 &= \FFp\{1,\mu_k,\mu_0\mu_k\} \subset E/J.
        \intertext{We now define}
        T_0 &= S_0\{\!\{ P(R)Q(\epsilon) \}\!\} \subset \Gamma/J^2\Gamma, \\
        T_1 &= S_1\{\!\{ P(R)Q(\epsilon) \}\!\} \subset \Gamma/J\Gamma.
      \end{align*}
      \begin{lemma}
        This defines a crossed algebra $T_\bullet\subset (P/J^2)_\bullet$
        as claimed in Theorem \ref{appmainthm}.
      \end{lemma}
      \begin{proof}
        Lemma \ref{ebpmodisq}
        shows that $(S_0,S_0[t_k,\tau_k])$ is a sub Hopf algebroid
        of $(E/J^2,\Gamma\us/J^2)$
        with $\Gamma\us/J^2 = E/J^2\otimes_{S_0} S_0[t_k,\tau_k]$.
        Therefore 
        $$T_0=\Hom_{S_0}(S_0[t_k,\tau_k],S_0)
        \hookrightarrow \Hom_{E/J^2}(\Gamma\us/J^2,E/J^2)=\Gamma/J^2$$
        is the inclusion of a subalgebra. 
        By Lemma \ref{partialtau}, $T_0$ is actually contained
        in $(P/J^2)_0 = \ker \partial : \Gamma/J^2 \rightarrow \Gamma/J^3$.
        The remaining details are left to the reader.
      \end{proof}

      To prove the Theorem it only remains to establish the weak equivalence
      $B_\bullet\rightarrow T_\bullet$. 
      Recall that $B_0$ is the free $\ZZ/p^2$-algebra 
      on generators $Q_0$ and $P^k$, $k\ge 1$.
      We can therefore define a multiplicative $\pmap_0:B_0\rightarrow T_0$
      via $Q_0 \mapsto Q(1)$ and $P^k\mapsto P(k)$.
      \begin{lemma}
        There is a weak equivalence $\pmap:B_\bullet \rightarrow T_\bullet$
        that extends $\pmap_0$.
      \end{lemma}
      \begin{proof}
        The multiplication on $\Gamma\us$ dualizes to a coproduct
        $\Delta_\Gamma:\Gamma\rightarrow \Gamma\totimes_E\Gamma$
        were $\totimes_E$ denotes a suitably completed tensor product.
        This turns $\Gamma$ into a topological Hopf algebra over $E$.
        We define the completed folding product $(P\cfotimes_E P)_\bullet$
        as the pullback
        \begin{align*}
          \xymatrix{
            A\otimes A \ar@{ >->}[r] 
            & \sfrac{\left(\Gamma\totimes_E\Gamma\right)}{\im \partial_\otimes}
            \ar[r]^-{\partial_\otimes} 
            & \ker\partial_\otimes \ar@{ ->>}[r] 
            & A\otimes A \\
            A\otimes A\ar@{=}[u] \ar@{ >->}[r] 
            & (P\cfotimes_E P)_1\ar@{ >->}[u] 
            \ar[r]^-{\partial_\otimes} 
            & P_0 \totimes_E P_0\ar@{ >->}[u] \ar@{ ->>}[r] 
            & A\otimes A \ar@{=}[u]\\            
          }
        \end{align*}
        where $\partial_\otimes = \partial\otimes\id + \id\otimes\partial$
        is the differential on $\Gamma\totimes_E\Gamma$.
        $\Delta_\Gamma$ then restricts to a coproduct 
        $\Delta_\bullet:P_\bullet \rightarrow (P\cfotimes_E P)_\bullet$.
        %secondary Hopf algebra structure
        %on $P_\bullet$. 
        Note that $\Delta_1$ is bilinear and symmetric, 
        since this is true for $\Delta_\Gamma$.
        By restriction we get a 
        $\Delta_\bullet:T_\bullet\rightarrow (T\cfotimes_S T)_\bullet$
        where the right hand side is given by
        \begin{align*}
          (T\cfotimes_S T)_0 &= S_0\{\!\{ P(R_1)Q(\epsilon_1)\otimes P(R_2)Q(\epsilon_2)\}\!\} \subset (P\cfotimes_E P)_1/J^2, \\
          (T\cfotimes_S T)_1 &= S_1\{\!\{ P(R_1)Q(\epsilon_1)\otimes P(R_2)Q(\epsilon_2)\}\!\} \subset (P\cfotimes_E P)_1/J.
        \end{align*}
        Let $\pmap\os T_\bullet$ be the pullback of $T_\bullet$ along $B_0\rightarrow T_0$.
        It inherits a secondary Hopf algebra structure
        from $T_\bullet$. This structure has $L=S=0$ since 
        the same is true for $P_\bullet$.
        Baues's Uniqueness Theorem thus implies $B_\bullet\cong \pmap\os T_\bullet$.
      \end{proof}
    \end{section}
  \end{appendix}

%%%%%%%%%%%%%%%%%%%%%%%%%%%%%%%%%%%%%%%%%%%%%%%%%%%%%%%%%%%%%%%%%%%%%%%%%%%%%%
%%%%%%%%%%%%%%%%%%%%%%%%%%%%%%%%%%%%%%%%%%%%%%%%%%%%%%%%%%%%%%%%%%%%%%%%%%%%%%
%%%%%%%%%%%%%%%%%%%%%%%%%%%%%%%%%%%%%%%%%%%%%%%%%%%%%%%%%%%%%%%%%%%%%%%%%%%%%%
%%%%%%%%%%%%%%%%%%%%%%%%%%%%%%%%%%%%%%%%%%%%%%%%%%%%%%%%%%%%%%%%%%%%%%%%%%%%%%

%\begin{section}{Bibliography}
  %\bibliographystyle{alpha}
  %\bibliography{dissbib}
%\end{section}

\end{document}